\documentclass[10pt]{article}

\newcommand{\hide}[1]{}

\usepackage[a4paper,margin=1in]{geometry}
\usepackage{amsmath,amssymb,amsthm,mathtools}
\usepackage{microtype}
\usepackage[skip=0.45\baselineskip plus 2pt,indent=0pt]{parskip}
\usepackage{tikz}
\usepackage{scalerel}
\usepackage[colorlinks=true,linkcolor=blue,citecolor=blue,urlcolor=blue]{hyperref}

\numberwithin{equation}{section}

\newtheorem{theorem}{Theorem}[section]
\newtheorem{lemma}[theorem]{Lemma}
\newtheorem{proposition}[theorem]{Proposition}
\newtheorem{corollary}[theorem]{Corollary}
\newtheorem{construction}[theorem]{Construction}

\newtheorem{problem}[theorem]{Problem}
\newtheorem{claim}[theorem]{Claim}
\newtheorem{fact}[theorem]{Fact}

\DeclareMathOperator{\PSL}{PSL}
\DeclareMathOperator{\PGL}{PGL}
\DeclareMathOperator{\Cay}{Cay}
\DeclareMathOperator{\ex}{ex}

\newcommand{\R}{\mathbb{R}}
\newcommand{\F}{\mathbb{F}}
\newcommand{\N}{\mathbb{N}}

\newsavebox\vvvbox
\savebox\vvvbox{\tikz{
        \draw[black,fill=black] (90:1) circle (.35);
        \draw[black,fill=black] (210:1) circle (.35);
        \draw[black,fill=black] (330:1) circle (.35);
        \draw[opacity=0] (0:1.2) circle (0.1);
}}
\newcommand{\vvv}{\mathord{\scaleobj{1.2}{\scalerel*{\usebox{\vvvbox}}{x}}}}
\newcommand{\pivvv}{\pi_{\vvv}}
\newcommand{\fin}{\mathrm{fin}}
\newcommand{\eps}{\varepsilon}

\title{\Large\bf Intervals of hypergraph Tur\'an densities}
\author{%
Xizhi Liu\thanks{\scriptsize School of Mathematical Sciences, University of Science and Technology of China, Hefei, China.
Email:~\texttt{liuxizhi@ustc.edu.cn}.}
\and
Oleg Pikhurko\thanks{\scriptsize Mathematics Institute and DIMAP, University of Warwick, Coventry CV4 7AL, UK.
Email:~\texttt{pikhurko@gmail.com}.}}
\date{\today}

\begin{document}
\maketitle

\begin{abstract}
    We prove that, for every integer $r\ge 3$, the set $\Pi^{(r)}_\infty$ of Tur\'an densities of (possibly infinite) families of $r$-graphs contains non-degenerate intervals, including an interval of the form $[1-\delta_r,1]$ for some $\delta_{r}>0$. This answers a question of Frankl, Peng, R\"odl and Talbot from 2007. This also shows that the Hausdorff dimension of $\Pi^{(r)}_\infty$ has the maximum possible value 1, thus resolving a question of Grosu from 2016, whereas previously it was not even known whether it is non-zero.
We also derive that the set 
of uniform Tur\'an densities of finite families of $3$-graphs is dense in a non-degenerate interval. 

\hide{    
The proof starts with a sequence of $3$-graphs $P_i$ whose Lagrangians converge slowly from above to a point $\alpha$, using 
the Ramanujan Cayley graphs constructed by Lubotzky--Phillips--Sarnak. We then incorporate this slowly convergent sequence into a suitable recursive construction, so that infinite subsequences can be used to approximate any prescribed real number in an interval.}
\end{abstract}
\section{Introduction}

Given an integer $r\ge 2$, an $r$-uniform hypergraph, or simply an $r$-graph, is a collection of $r$-subsets of a vertex set.  For a hypergraph $H$, let $V(H)$ denote its vertex set, and write $v(H)\coloneqq |V(H)|$ for its number of vertices. We identify a hypergraph with its edge set, and hence write $|H|$ for the number of edges of $H$.

Given a family $\mathcal F$ of $r$-graphs, let $\ex(n,\mathcal F)$ be the maximum number of edges in an $\mathcal F$-free $r$-graph on $n$ vertices. The \emph{Tur\'an density} of $\mathcal F$ is
\[
        \pi(\mathcal F) \coloneqq \lim_{n\to\infty}{\ex(n,\mathcal F)}/{\tbinom nr}.
\]
The existence of this limit follows from the averaging argument of Katona, Nemetz and Simonovits~\cite{KatonaNemetzSimonovits1964}. Let $\Pi^{(r)}_{\fin}$ denote the set of Tur\'an densities of finite families of $r$-graphs, and let $\Pi^{(r)}_\infty$ denote the corresponding set when the forbidden family is allowed to be infinite. We allow $\mathcal F$ to be empty, and hence $1\in \Pi^{(r)}_{\fin}$.

For graphs, the seminal Erd\H{o}s--Stone--Simonovits theorem~\cite{ErdosStone1946,ErdosSimonovits1966} determines all possible Tur\'an densities:
\[
        \Pi^{(2)}_{\fin} = \Pi^{(2)}_\infty = \{1-1/k \colon k\in\mathbb N,\ k\ge 1\} \cup \{1\}.
\]
For $r\ge3$, determining $\pi(\mathcal F)$ for a given family $\mathcal F$ is notoriously hard in general. This problem, for example,  for the complete  $\ell$-vertex $r$-graph $K_\ell^{(r)}$, raised in~\cite{Turan1941}, remains wide open.
 In fact, the $\$500$ reward of  Erd\H{o}s, see e.g.~\cite{Erd84}, for determining $\pi(K_\ell^{(r)})$ in at least one case with $\ell>r\ge3$ is still unclaimed. For more background on hypergraph Tur\'an problems, see the surveys of F{\"u}redi~\cite{Furedi1991}, Sidorenko~\cite{Sidorenko1995}, and Keevash~\cite{Keevash2011}.

The study of the algebraic and topological properties of the sets $\Pi^{(r)}_{\fin}$ and $\Pi^{(r)}_{\infty}$ for $r\ge 3$ goes back to the classical work of Erd\H{o}s~\cite{Erdos1964}. We call a number $\alpha\in[0,1)$ a \emph{jump} for $r$-graphs if there is a constant $c>0$ such that $\Pi^{(r)}_\infty\cap(\alpha,\alpha+c)=\emptyset$. Otherwise, we call $\alpha$ a \emph{non-jump}. In~\cite{Erdos1964} Erd\H{o}s proved that $\Pi^{(r)}_\infty\cap(0,r!/r^r)=\emptyset$ for every $r\ge3$ and asked whether every number is a jump. This became one of his famous $\$1000$ problems. Frankl and R\"odl~\cite{FranklRodl1984} famously disproved this by proving the existence of non-jumps for every $r\ge3$.

The topological viewpoint on possible Tur\'an densities was developed by the second author in~\cite{Pikhurko2014}, where it was proved that  $\Pi^{(r)}_\infty=\overline{\Pi^{(r)}_{\fin}}$ (in particular, $\Pi^{(r)}_\infty$ is closed) and that $\Pi^{(r)}_\infty$ has cardinality continuum for every $r\ge3$. On the algebraic side, for some time it was open whether $\Pi^{(r)}_{\fin}$ could contain an irrational number; see the conjecture of Chung and Graham~\cite[p.~95]{ChungGraham1998}. Such examples were independently found by Baber and Talbot~\cite{BaberTalbot2012} and by the second author~\cite{Pikhurko2014}.
Grosu~\cite{Grosu2016} did a systematic study of algebraic properties of $\Pi^{(r)}_{\fin}$ and $\Pi^{(r)}_\infty$ and asked whether, for some $r\ge3$, the set $\Pi^{(r)}_{\fin}$ contains an algebraic number of degree greater than $r-1$. This was answered in a strong form by the authors~\cite{LiuPikhurko2023}, who showed that, for every $r\ge3$ and every integer $d$, the set $\Pi^{(r)}_{\fin}$ contains a number of algebraic degree at least $d$. Further work on gaps in $\Pi_{\infty}^{(3)}$ includes Baber and Talbot's flag algebra proof of explicit jumps in $\Pi^{(3)}_\infty$~\cite{BaberTalbot2011}, based on Razborov's method~\cite{Razborov2007} and a criterion of Frankl and R\"odl; see also~\cite{BrownSimonovits1984,Pikhurko2015}. Codegree and $\ell$-degree versions of Tur\'an density were studied by Mubayi and Zhao~\cite{MubayiZhao07} and by Lo and Markstr\"om~\cite{LoMarkstrom14}; in that setting, the set of densities was shown to form a dense set in $[0,1]$. The multigraph analogue of the jumping constant conjecture was studied by R\"odl and Sidorenko~\cite{RodlSidorenko95}. Further work on higher-uniformity non-jumps and jumping densities includes e.g.~\cite{Peng07NonJumping4Uniform,Peng09JumpingDensities,Peng2008I,Peng2007II,HouLiYangZhang24}.

Despite extensive study of possible Tur\'an densities, the fundamental question of whether $\Pi^{(r)}_\infty$ contains any interval of positive length for $r\ge3$ remained open. More specifically, Frankl, Peng, R\"odl and Talbot~\cite[Question~6.1]{FPRT2007} asked whether, for every $r\ge3$, there is $\alpha_r<1$ such that the whole terminal interval $[\alpha_r,1]$ lies in $\Pi^{(r)}_\infty$. Grosu~\cite[Problem~5]{Grosu2016} asked to determine the Hausdorff dimension of $\Pi^{(r)}_\infty$ or even just to prove that it is non-zero.

Our main result settles these questions. In particular, it proves that, for every uniformity $r\ge3$, the set of possible Tur\'an densities contains a terminal interval. Consequently, $\Pi^{(r)}_\infty$ has positive Lebesgue measure, and has Hausdorff dimension $1$.

\begin{theorem}\label{thm:main} For every integer $r\ge 3$, there exists $\delta_r>0$ such that $[1-\delta_r,1]\subseteq \Pi^{(r)}_\infty$. \end{theorem}

Although Theorem~\ref{thm:main} is stated qualitatively, our proof is quantitative in principle: for each fixed $r$, it gives an explicit positive value of $\delta_r$ once all auxiliary parameters are chosen explicitly. We make no attempt to optimize this value. In particular, the need to take a sufficiently far tail of primes in a fixed arithmetic progression, using explicit estimates for primes in arithmetic progressions such as those in~\cite{BennettMartinOBryantRechnitzer2018}, causes a large loss, so the resulting $\delta_r$ is very small. We discuss this further in Section~\ref{sec:concluding-remarks}.

We briefly indicate the main ingredients of the proof of Theorem~\ref{thm:main}. First, we construct a slowly convergent sequence of finite $3$-graphs $P_i$ whose Lagrangians approach a limit $\alpha$ from above, with successive gaps to the limit asymptotically equal (Theorem~\ref{thm:smooth-sequence}). 
This sequence is built from a $3$-partite spectral gadget arising from the Ramanujan regular graphs constructed by  Lubotzky--Phillips--Sarnak (LPS), with the properties used by us being summarised in Theorem~\ref{thm:LPS}. We then incorporate it into a recursive construction of possible Tur\'an densities (Section~\ref{sec:recursive-interval}). The near equality of the successive gaps makes the recursive tails overlap, and an infinite subsequence can be chosen so as to approximate any prescribed real number in a non-degenerate interval. This fills an interval of densities in uniformity $3$. Finally, a standard complete join operation moves such an interval to a terminal interval near $1$, and a standard lifting argument transfers the result to every uniformity $r\ge3$.

\bigskip 

There is a parallel uniformly dense Tur\'an theory for $3$-graphs, initiated by Erd\H{o}s and S\'os~\cite{ErdosSos1982}, in which the host $3$-graph is required to be dense not only globally, but also on every vertex subset of linear size.
Given $d\in[0,1]$ and $\eta>0$, we say that a $3$-graph $H$ on $n$ vertices is \emph{$(d,\eta,\vvv)$-uniformly dense} (or just \emph{$(d,\eta,\vvv)$-dense}) if
\[
        |H[U]|\ge d\tbinom{|U|}{3}-\eta n^3 \qquad\text{for every } U\subseteq V(H).
\]
The subscript $\vvv$ records that the density condition is imposed on vertex subsets. The additive error makes the condition robust on all linear-sized vertex subsets. For a family $\mathcal F$ of $3$-graphs, its \emph{uniform Tur\'an density} is
\[
        \pivvv(\mathcal F)
        \coloneqq
        \sup\left\{
        d\in[0,1]\colon
        \begin{array}{l}
        \text{for every }\eta>0\text{ and }n_0\in\N\text{, there exists an}\\
        \mathcal F\text{-free }(d,\eta,\vvv)\text{-dense }3\text{-graph }H
        \text{ with }v(H)\ge n_0
        \end{array}
        \right\}.
\]
Equivalently, every density strictly larger than $\pivvv(\mathcal F)$ forces a member of $\mathcal F$ in all sufficiently large uniformly dense host $3$-graphs. We write
\[
        \Pi_{\vvv,\fin} \coloneqq \{\pivvv(\mathcal F)\colon \mathcal F\text{ is a finite family of }3\text{-graphs}\}.
\]
Among their motivating questions were the determination of the uniform Tur\'an densities of $K_4^{(3)-}$ and $K_4^{(3)}$, where $K_4^{(3)-}$ denotes the $3$-graph on four vertices with three edges. The former was settled by Glebov, Kr\'al' and Volec~\cite{GlebovKralVolec2016}, who showed that $\pivvv(K_4^{(3)-})=1/4$ using a computer-generated proof via flag algebras. The value of $\pivvv(K_4^{(3)})$ remains open. In a series of foundational works, Reiher, R\"odl and Schacht~\cite{ReiherRodlSchacht2018,ReiherRodlSchacht2018Mantel,ReiherRodlSchacht2018Vanishing,Reiher2020,Schacht2022ICM} developed embedding methods for uniformly dense hypergraphs via the hypergraph regularity method~\cite{Gowers2006,RodlSkokan2004,NagleRodlSchacht2006}, in particular re-proving the above result on $K_4^{(3)-}$ without using a computer. 

There have been recent studies on the structure of possible uniform Tur\'an densities. Reiher, R\"odl and Schacht~\cite{ReiherRodlSchacht2018Vanishing} characterized the $3$-graphs with vanishing uniform Tur\'an density and showed that $0$ is followed by a gap: specifically, there is no uniform Tur\'an density in the interval $(0,1/27)$. Garbe, Kr\'al' and Lamaison~\cite{GarbeKralLamaison2024} proved that this first jump is sharp by constructing $3$-graphs with uniform Tur\'an density exactly $1/27$. More recently, the palette viewpoint of Lamaison~\cite{Lamaison2024Palettes} and King, Piga, Sales and Sch\"ulke~\cite{KingPigaSalesSchulke2025} has connected classical non-jumps with non-jumps in the set of finite-family uniform Tur\'an densities.

Our construction also has a uniformly dense analogue for finite forbidden families. This follows from the density of $3$-graph Lagrangians in the interval constructed here, together with the theorem of King, Piga, Sales and Sch\"ulke~\cite{KingPigaSalesSchulke2025} that every finite-palette Lagrangian is realized as the uniform Tur\'an density of a finite forbidden family.

\begin{corollary}\label{cor:dense-uniform-finite-interval} There exist real numbers $\alpha < \beta$ such that $[\alpha, \beta]\subseteq \overline{\Pi_{\vvv,\fin}}$. Equivalently, $\Pi_{\vvv,\fin}$ is dense in a non-degenerate interval. \end{corollary}

The paper is organized as follows. Section~\ref{sec:preliminaries} collects preliminaries on Lagrangians and the Ramanujan Cayley graphs used in the spectral construction. Section~\ref{sec:smooth-sequence-construction} constructs the smooth decreasing sequence of $3$-graphs. Section~\ref{sec:recursive-interval} uses this sequence in a recursive construction to obtain an interval of possible $3$-uniform Tur\'an densities. Section~\ref{sec:extensions-uniform-analogue} turns this into a terminal interval, lifts the result to all uniformities $r\ge3$, and derives the uniformly dense finite-family corollary via finite palettes. Section~\ref{sec:concluding-remarks} contains concluding remarks.

For an integer $t\ge0$, we denote $[t]\coloneqq\{1,\ldots,t\}$, with $[0]=\emptyset$.

We use standard asymptotic notation throughout. Thus $f\sim g$ means $f/g\to1$, and $a_i\downarrow a$ means that the sequence $(a_i)_{i\ge 1}$ decreases monotonically to $a$, staying strictly above $a$. The symbol $\to$ refers to the limiting process specified by the surrounding statement, while $t\to0^+$ means that $t$ tends to $0$ through positive values. In extremal statements, the number of vertices tends to infinity. In the spectral construction, once the prime $p$ has been fixed, all asymptotics involving the LPS graphs are taken as $j\to\infty$, equivalently as $q_j\to\infty$ and $n_j=|\PSL_2(\F_{q_j})|\to\infty$. The parameters $D=p+1$, $\theta$, $k$, and the recursive parameter $m$ are then regarded as fixed. Constants implicit in $O(\cdot)$, $o(\cdot)$, and $\Theta(\cdot)$ may depend on fixed parameters, but not on the variable tending to infinity. Limits in small positive parameters, such as $\alpha\to0^+$ or $\eta\to0^+$, are always stated explicitly. 

\section{Preliminaries}
\label{sec:preliminaries}
\subsection{Lagrangian}
For a finite $r$-graph $G$, write $\varrho(G)\coloneqq |G|/\binom{v(G)}{r}$ for its \emph{edge density}.

A \emph{blow-up} of a finite $r$-graph $G$ is obtained by replacing each vertex $v\in V(G)$ by a non-empty vertex class $V_v$, and replacing each edge $\{v_1,\ldots,v_r\}\in G$ by all $r$-sets with one vertex in each of $V_{v_1},\ldots,V_{v_r}$.

For an integer $n\ge 0$, the standard \emph{$n$-dimensional simplex} is 
\[
        \Delta_n \coloneqq \left\{ \mathbf{x}=(x_1,\ldots,x_{n+1})\in\R_{\ge 0}^{n+1}\colon x_1 + \cdots + x_{n+1} = 1 \right\}
\]
More generally, for a finite set $V$, write
\[
        \Delta_V\coloneqq \left\{\mathbf{y}\in\R_{\ge 0}^{V}\colon\ \sum_{v\in V}y_v=1\right\}.
\]
The \emph{Lagrangian polynomial} of a finite $r$-graph $G$ is
\[
        P_G(\mathbf{y})\coloneqq r!\sum_{e\in G}\prod_{v\in e}y_v \qquad\text{for all } \mathbf{y}\in\R^{V(G)}.
\]
Its \emph{Lagrangian} is
\[
        \lambda(G)\coloneqq \max\{P_G(\mathbf{y})\colon\ \mathbf{y}\in\Delta_{V(G)}\}.
\]

The set of Lagrangians of finite $r$-graphs is defined as 
\[
        \Lambda^{(r)} \coloneqq \{\lambda(G)\colon G\text{ is a finite }r\text{-graph}\}.
\]

\begin{theorem}[{\cite[Corollary~4 and Theorem~24]{Pikhurko2014}}]\label{thm:pikhurko-lagrangian-densities} For every integer $r\ge 2$, we have $\Lambda^{(r)}\subseteq \Pi^{(r)}_{\fin}$. Moreover, for every integer $r\ge 3$, the set $\Lambda^{(r)}$ is dense in $\Pi^{(r)}_\infty$, that is, $\overline{\Lambda^{(r)}} = \Pi^{(r)}_{\infty}$. 
\end{theorem}

We also recall the related classical result of Brown and Simonovits~\cite{BrownSimonovits1984}: the Tur\'an density of an infinite forbidden family is the limit of the Tur\'an densities of its finite subfamilies. In particular, $\Pi^{(r)}_\infty\subseteq\overline{\Pi^{(r)}_{\fin}}$.

We shall use the following consequence of Pikhurko's limit characterization~\cite[Theorem 24]{Pikhurko2014}.

\begin{theorem}[{\cite[Theorem~24]{Pikhurko2014}}]\label{thm:tight} Let $r\ge 3$ and $x\in[0,1]$. Then $x\in \Pi^{(r)}_\infty$ if and only if there is a sequence of $r$-graphs $G_n$, with $v(G_n)\to\infty$, such that $\varrho(G_n)\to x$ and $\lambda(G_n)\to x$. \end{theorem}

\subsection{Ramanujan graphs and Cayley graphs}

Even though some of the definitions and facts that we state and discuss are standard,
we decided to provide details for readers who may be unfamiliar with them.

Let us recall the definitions of some standard matrix groups over the field $\F_q$ for an odd prime power $q$:

\[
        \begin{aligned}
        \operatorname{GL}_2(\F_q)&\coloneqq \{A\in\F_q^{2\times 2}\colon\det A\ne 0\},&
        \PGL_2(\F_q)&\coloneqq \operatorname{GL}_2(\F_q)/\F_q^\times I,\\
        \operatorname{SL}_2(\F_q)&\coloneqq \{A\in\F_q^{2\times 2}\colon\det A=1\},&
        \PSL_2(\F_q)&\coloneqq \operatorname{SL}_2(\F_q)/\{\pm I\}.
        \end{aligned}
\]
Here (and later) $I$ denotes the identity matrix of the appropriate dimension and $\F_q^\times$ denotes the multiplicative group on the non-zero elements of~$\F_q$. Thus, for example,  $\F_q^\times I=\{tI\colon t\in\F_q^\times\}$ is the group of non-zero scalar multiples of the identity matrix. The groups $\PGL_2(\F_q)$ and $\PSL_2(\F_q)$ are called the \emph{projective general linear group} and the \emph{projective special linear group}, respectively. Recall that $|\PSL_2(\F_q)|=q(q^2-1)/2$.

If $\Gamma$ is a finite group and $S=S^{-1}\subseteq\Gamma\setminus\{1_\Gamma\}$, then the \emph{Cayley graph} $\Cay(\Gamma,S)$ is the graph with vertex set $\Gamma$ in which $g,h\in \Gamma$ are adjacent if and only if $h=gs$ for some $s\in S$.

For an odd prime $q$ and an integer $a$, the \emph{Legendre symbol} $\left(\frac{a}{q}\right)$ is defined by
\[
        \left(\frac{a}{q}\right)
        =
        \begin{cases}
        0, & q\mid a,\\
        1, & q\nmid a \text{ and there is an } x\in\mathbb F_q^\times
        \text{ with } x^2\equiv a\pmod q,\\
        -1, & q\nmid a \text{ and no such } x \text{ exists}.
        \end{cases}
\]
Thus, for $q\nmid a$, the condition $\left(\frac{a}{q}\right)=1$ means that $a$ is a non-zero square modulo $q$, while $\left(\frac{a}{q}\right)=-1$ means that $a$ is a non-square modulo $q$.

For a finite graph $G$, its \emph{adjacency matrix} $A_G$ is the matrix whose $(u,v)$-entry is $1$ if $u$ and $v$ are adjacent, and is $0$ otherwise. A  $D$-regular graph $G$ is called \emph{Ramanujan} if every eigenvalue of $A_G$ other than the trivial eigenvalue $D$ has absolute value at most $2\sqrt{D-1}$.
Recall that the bound $2\sqrt{D-1}$ is asymptotically best possible for $D$-regular graphs by the Alon--Boppana theorem; see, for example, Alon~\cite{Alon1986} and Nilli's sharpened form~\cite{Nilli1991}.

The seminal construction of Lubotzky, Phillips and Sarnak (the LPS construction)~\cite{LPS1988} gives Ramanujan graphs. Related Ramanujan graph constructions were independently obtained by Margulis~\cite{Margulis1988}. We will use the LPS construction, which defines the graphs
via quaternions as follows. Fix a prime $p\equiv1\pmod4$ and choose primes $q\equiv1\pmod4$, $q\neq p$, for which $p$ is a quadratic residue modulo $q$, i.e., $\left(\frac{p}{q}\right) = 1$. The $p+1$ generators come from the integral representations
\[
        p=a_0^2+a_1^2+a_2^2+a_3^2, \qquad a_0>0,\quad a_0\text{ odd},\quad\text{and}\quad a_1,a_2,a_3\text{ even}.
\]
After choosing $\iota\in\mathbb F_q$ with $\iota^2=-1$, each such quadruple gives the projective class of
\[
        \begin{pmatrix}
        a_0+\iota a_1 & a_2+\iota a_3\\
        -a_2+\iota a_3 & a_0-\iota a_1
        \end{pmatrix}.
\]
Since its determinant is $p$, this class lies in $\PSL_2(\mathbb F_q)$ when $p$ is a square modulo $q$ and thus a scalar multiple of the matrix has determinant $1$. Taking these projective classes as the generating set gives the Cayley graph used below.

More generally, for any distinct primes $p,q\equiv1\pmod4$, the LPS construction gives a $(p+1)$-regular Ramanujan graph $X^{p,q}$. The Legendre symbol $\left(\frac{p}{q}\right)$ determines the ambient projective group. If $\left(\frac{p}{q}\right)=1$, then the generators have square determinant and lie in $\PSL_2(\mathbb F_q)$, so $X^{p,q}$ is a Cayley graph on $\PSL_2(\mathbb F_q)$; this is the non-bipartite case. If $\left(\frac{p}{q}\right)=-1$, then the corresponding Cayley graph is on $\PGL_2(\mathbb F_q)$ and is bipartite. Both graphs are Ramanujan (in the latter case using the bipartite version of the property). We use only the first case.

\begin{theorem}[\cite{LPS1988}]\label{thm:LPSOneQ} Let $p<q$ be primes congruent to $1$ modulo $4$ with $\left(\frac{p}{q}\right)=1$, and put $\Gamma=\PSL_2(\mathbb F_q)$. Then there is a symmetric generating set $S\subseteq \Gamma$ with $|S|=p+1$ such that the Cayley graph $\Cay(\Gamma,S)$ is Ramanujan. Consequently, this graph is connected and non-bipartite.
\end{theorem}

The next lemma records the number-theoretic input needed to find an infinite sequence of suitable primes~$q$. The extra condition $q_{j+1}/q_j\to1$ will later let the group sizes vary smoothly.

\begin{lemma}\label{lem:choice-qj}
Let $p\equiv 1 \pmod 4$ be a prime. Then there is an increasing sequence of primes $q_1<q_2<q_3<\cdots$ such that for $j \ge 1$,
\[
        q_j\equiv 1 \pmod 4,\qquad \left(\frac{p}{q_j}\right)=1,\qquad q_j\to\infty,\qquad\text{and}\qquad \frac{q_{j+1}}{q_j}\to 1.
\]
In particular, each pair $(p,q_j)$ satisfies the hypotheses for the $\PSL_2(\mathbb F_{q_j})$ case of the LPS construction.
\end{lemma}
\begin{proof}
Let $q$ be a prime satisfying $q\equiv 1 \pmod {4p}$. Then $q\ne p$, since $q\equiv1\pmod p$. Clearly $q\equiv 1\pmod 4$. Also $q\equiv 1\pmod p$, and therefore $\left(\frac{q}{p}\right)=\left(\frac{1}{p}\right)=1$. By quadratic reciprocity,
\[
        \left(\frac{p}{q}\right)\left(\frac{q}{p}\right) = (-1)^{\frac{p-1}{2}\frac{q-1}{2}}.
\]
Since $p\equiv 1\pmod 4$, the integer $(p-1)/2$ is even. Hence the sign on the right-hand side is $+1$, and so
\[
        \left(\frac{p}{q}\right)=\left(\frac{q}{p}\right)=1.
\]
Thus $p$ is a quadratic residue modulo every prime $q\equiv 1\pmod {4p}$.

It remains to choose such primes with consecutive ratio tending to $1$. By the prime number theorem in arithmetic progressions due to de la Vall\'ee Poussin, see e.g.~\cite[Chapter~20]{Davenport2000},
\[
        \pi(x;4p,1) \coloneqq \#\{q\le x\colon q\text{ prime and }q\equiv 1\pmod {4p}\} \sim \frac{\operatorname{Li}(x)}{\varphi(4p)}.
\]
Here $\varphi$ denotes Euler's totient function, so $\varphi(4p)=\#\{1\le a\le 4p\colon\gcd(a,4p)=1\}$, and $\operatorname{Li}(x)\coloneqq \int_2^x \frac{dt}{\log t}$ is the logarithmic integral. In particular, there are infinitely many primes in this progression. Let $q_1<q_2<q_3<\cdots$ be their increasing enumeration.

Fix $\eta>0$. The same asymptotic gives
\[
        \pi((1+\eta)x;4p,1)-\pi(x;4p,1) \sim \frac{\operatorname{Li}((1+\eta)x)-\operatorname{Li}(x)} {\varphi(4p)} \sim \frac{\eta x}{\varphi(4p)\log x}.
\]
This tends to infinity as $x\to\infty$. Hence, for all sufficiently large $x$, the interval $(x,(1+\eta)x]$ contains a prime congruent to $1\pmod {4p}$. Taking $x=q_j$ for sufficiently large $j$, we get $q_{j+1}\le (1+\eta)q_j$. Since always $q_{j+1}>q_j$, it follows that
\[
        1\le \liminf_{j\to\infty}\frac{q_{j+1}}{q_j} \le \limsup_{j\to\infty}\frac{q_{j+1}}{q_j} \le 1+\eta.
\]
As $\eta>0$ was arbitrary, we have ${q_{j+1}}/{q_j}\to 1$. Thus each pair $(p,q_j)$ consists of distinct primes with $p\equiv q_j\equiv1\pmod4$ and $\left(\frac{p}{q_j}\right)=1$, which are exactly the hypotheses needed for the $\PSL_2(\mathbb F_{q_j})$ case of the LPS construction.
\end{proof}

Let us state the properties that we will use in our proofs and that directly follow from Lemma~\ref{lem:choice-qj} and Theorem~\ref{thm:LPSOneQ}.

\begin{theorem}\label{thm:LPS} Let $p\equiv 1\pmod 4$ be a prime and put $D \coloneqq p+1$. Then there is an increasing sequence $(q_j)_{j\ge 1}$ of primes and a sequence $(S_j)_{j\ge 1}$ of sets such that  $q_j\to\infty$ and $q_{j+1}/q_j\to 1$ as $j\to\infty$ and, for every $j\ge 1$, it holds that
$q_j\equiv 1\pmod 4$, $\left(\frac{p}{q_j}\right)=1$, $S_j$ is a  symmetric generating set in $\Gamma_j\coloneqq \PSL_2(\mathbb F_{q_j})$, $|S_j|=D$, and the $D$-regular Cayley graph $\Cay(\Gamma_j,S_j)$ is Ramanujan.
\end{theorem}

For functions $f,g$ on a finite set $U$, we write
\[
        \langle f,g\rangle_U\coloneqq \sum_{u\in U}f(u)g(u), \qquad \|f\|_{2,U}\coloneqq \sqrt{\langle f,f\rangle_U}.
\]
When $U$ is clear from context, we omit it from the notation.

If $\Gamma$ is a finite group and $S=S^{-1}\subseteq\Gamma\setminus\{1_\Gamma\}$, the adjacency matrix $M=M_{\Gamma,S}$ of $\Cay(\Gamma,S)$ acts on functions $h\colon\Gamma\to\R$ by
\[
        (Mh)(u)\coloneqq \sum_{s\in S}h(us) \qquad\text{for all } h\colon\Gamma\to\R\text{ and }u\in\Gamma.
\]
The following is the standard expander mixing lemma in a functional form; see, for example, \cite{Alon1986,Lubotzky2012}. For completeness, we include the short proof.

\begin{lemma}\label{lem:mixing}
Let $\Gamma$ be a finite group, let $S=S^{-1}\subseteq\Gamma\setminus\{1_\Gamma\}$, and put $D\coloneqq |S|$ and $n\coloneqq |\Gamma|$. Suppose that $\Cay(\Gamma,S)$ is connected and non-bipartite, and that every non-trivial adjacency eigenvalue has absolute value at most $\theta$. Let $M=M_{\Gamma,S}$. 
Then every function $h\colon\Gamma\to\R$ with $\sum_{u\in\Gamma}h(u)=0$ satisfies
\begin{equation}
        \|Mh\|_2\le \theta\|h\|_2 .
        \label{eq:spectral-L2-bound}
\end{equation}
Moreover, for any functions $f,g\colon\Gamma\to\R$,
\begin{equation}
        \langle f,Mg\rangle \le Dn\,\bar f\,\bar g+\theta\|f^\circ\|_2\|g^\circ\|_2,
        \label{eq:mixing}
\end{equation}
where $\mathbf 1$ denotes the constant-one function on $\Gamma$, and for any function $h\colon\Gamma\to\R$ we write
\[
        \bar h\coloneqq \frac1n\sum_{u\in\Gamma}h(u), \qquad h^\circ\coloneqq h-\bar h\,\mathbf 1 .
\]
\end{lemma}

\begin{proof}
Since $S=S^{-1}$, the Cayley graph is undirected and $M$ is self-adjoint. Indeed,
\[
        \langle Mf,h\rangle = \sum_{\{u,v\}\in\Cay(\Gamma,S)} \bigl(f(v)h(u)+f(u)h(v)\bigr) = \langle f,Mh\rangle .
\]
Thus $M$ has an orthonormal basis of real eigenvectors.

Since $\Cay(\Gamma,S)$ is $D$-regular, the constant function $\mathbf 1$ is an eigenvector with eigenvalue $D$, that is, $M\mathbf 1=D\mathbf 1$. Since the graph is connected, this eigenspace is exactly the one-dimensional space of constant functions. To see this directly, suppose that $M\xi=D\xi$. Let $u_0\in\Gamma$ be a vertex at which $\xi$ is maximal. Then
\[
        D\xi(u_0)=(M\xi)(u_0)=\sum_{v\sim u_0}\xi(v)\le D\xi(u_0).
\]
Equality forces $\xi(v)=\xi(u_0)$ for every neighbor $v$ of $u_0$. By connectedness, this propagates along paths and gives $\xi\equiv \xi(u_0)$.

Since the graph is non-bipartite, $-D$ is not an eigenvalue. Indeed, if $M\xi=-D\xi$ for some nonzero $\xi$, then choosing a vertex $u_0$ where $|\xi|$ is maximal gives $-D\xi(u_0)=\sum_{v\sim u_0}\xi(v)$. Equality in the triangle inequality forces every neighbor $v$ of $u_0$ to satisfy $\xi(v)=-\xi(u_0)$. Repeating this along paths shows that vertices at even distance from $u_0$ have value $\xi(u_0)$, while vertices at odd distance from $u_0$ have value $-\xi(u_0)$. This would define a bipartition of the graph, a contradiction.

Let $\mathbf v_0=\frac{1}{\sqrt n}\mathbf 1, \mathbf v_1, \ldots, \mathbf v_{n-1}$ be an orthonormal eigenbasis of $M$, with $M\mathbf v_i=\lambda_i \mathbf v_i$ for $0 \le i \le n-1$. Here $\lambda_0=D$, and all other eigenvalues are non-trivial. By the hypothesis, $|\lambda_i|\le \theta$ for all $i \ge 1$.

Now let $h\colon\Gamma\to\R$ satisfy $\sum_{u\in\Gamma}h(u)=0$. This condition is exactly $\langle h,\mathbf 1\rangle=0$, so $h$ is orthogonal to $\mathbf v_0$. Hence $h=\sum_{i\in[n-1]} c_i \mathbf v_i$ for some $(c_1, \ldots, c_{n-1})\in\R^{n-1}$. Applying $M$ gives $Mh=\sum_{i\in[n-1]} c_i\lambda_i \mathbf v_i$. By orthonormality,
\[
        \|Mh\|_2^2 = \sum_{i\in[n-1]}c_i^2\lambda_i^2 \le \theta^2\sum_{i\in[n-1]}c_i^2 = \theta^2\|h\|_2^2 .
\]
Taking square roots proves \eqref{eq:spectral-L2-bound}.

Observe that $M$ sends zero-sum functions to zero-sum functions: indeed, $\Cay(\Gamma,S)$ is $D$-regular, and for any $h\colon\Gamma\to\R$, since each map $u\mapsto us$ is a bijection of $\Gamma$, we have 
\[
        \sum_{u\in\Gamma}(Mh)(u)=\sum_{u\in\Gamma}\sum_{s\in S}h(us)=D\sum_{u\in\Gamma}h(u),
\]
Now let $f,g\colon\Gamma\to\R$. The functions $f^\circ$ and $g^\circ$ have sum zero, and since $M\mathbf 1=D\mathbf 1$, we have
\[
        \langle f,Mg\rangle = \langle \bar f\,\mathbf 1+f^\circ, M(\bar g\,\mathbf 1+g^\circ)\rangle = Dn\,\bar f\,\bar g+\langle f^\circ,Mg^\circ\rangle ,
\]
where the mixed terms vanish by the zero-sum property. By \eqref{eq:spectral-L2-bound}, $\|Mg^\circ\|_2\le\theta\|g^\circ\|_2$. Therefore Cauchy's inequality gives
\[
        \langle f^\circ,Mg^\circ\rangle \le |\langle f^\circ,Mg^\circ\rangle| \le \|f^\circ\|_2\|Mg^\circ\|_2 \le \theta\|f^\circ\|_2\|g^\circ\|_2 .
\]
This proves \eqref{eq:mixing}.
\end{proof}
\section{Construction of a smooth decreasing sequence}
\label{sec:smooth-sequence-construction}
The goal of this section is to prove Theorem~\ref{thm:smooth-sequence} below, which supplies a sequence of finite $3$-graphs whose Lagrangians approach a limit from above, with consecutive gaps to the limit being asymptotically equal. 

\begin{theorem}\label{thm:smooth-sequence}
There exist a real number $\alpha$ and a sequence of $3$-graphs $(P_i)_{i\ge 1}$ such that
\[
        \lambda(P_i)\downarrow \alpha\quad\text{and}\quad \frac{\lambda(P_i)-\alpha}{\lambda(P_{i+1})-\alpha}\to 1,\qquad\text{as $i\to\infty$}. 
\]
\end{theorem}

There are a few steps for the constructions of Theorem~\ref{thm:smooth-sequence}. 
We first build a $3$-partite spectral gadget from LPS Ramanujan graphs in Construction~\ref{con:spectral-gadget}; Lemma~\ref{lem:centered} gives the spectral estimate needed to control its Lagrangian. We then place copies of this gadget inside a complete multipartite $3$-graph, as defined in Construction~\ref{con:outer-Q}, and Proposition~\ref{prop:Q-lag} gives the required asymptotic formula for its Lagrangian. Finally, Lemma~\ref{lem:extraction}, together with the smooth growth of the underlying LPS groups, allows us to extract a decreasing subsequence whose successive gaps to the limit are asymptotically equal.

Throughout this section, we fix the following parameters.

Fix a sufficiently large prime $p\equiv 1\pmod 4$, and put
\[
        D\coloneqq p+1,\qquad \theta\coloneqq 2\sqrt{D-1}.
\]
Choose an integer $k$ such that
\begin{equation}
        \left(\frac97\right)^{1/3}k^{-2/3}<\frac3{\theta+3} \qquad\text{and}\qquad k<\frac{2D}{9}+1 .
        \label{eq:k-choice}
\end{equation}
Such a choice is possible for all sufficiently large $p$, because $\theta=O(\sqrt D)$, so the first inequality only requires $k\gg D^{3/4}$, while the second allows $k=\Theta(D)$.
Put
\[
        \alpha\coloneqq 1-\frac1{k^2}.
\]

Let $(\Gamma_j,S_j)_{j=1}^\infty$ be the sequence returned by Theorem~\ref{thm:LPS} on input $p$. Write
\[
        n_j\coloneqq |\Gamma_j|\qquad\text{and}\qquad m_j\coloneqq 3n_j .
\]
Since $q_j\to\infty$ and $q_{j+1}/q_j\to 1$, we have
\begin{align*}
    \frac{m_{j+1}}{m_j}
    = \frac{3n_{j+1}}{3n_{j}}
    = \frac{3q_{j+1}(q_{j+1}^{2}-1)/2}{3q_j(q_j^2-1)/2}
    \to 1. 
\end{align*}

\subsection{The internal spectral gadget}
Recall that $k$ was fixed to satisfy~\eqref{eq:k-choice} and that we let $\alpha:=1-1/k^2$.

\begin{construction}[The internal spectral gadget $H_j$]\label{con:spectral-gadget}
    For each $j$, define a $3$-partite $3$-graph $H_j$ with vertex classes $X_j,Y_j,Z_j$, where each class is a disjoint copy of $\Gamma_j$. We identify each of these copies with $\Gamma_j$ when writing the group operation. 
    Each triple $\{x,y,z\}$ with $(x,y,z) \in X_j\times Y_j \times Z_j$ forms an edge of $H_j$ if and only if $z=xsy$ for some $s \in S_j$. 
\end{construction}

Note that for every $j$, we have $|V(H_j)|=m_j=3n_j$ and $|H_j|=Dn_j^2$.

\begin{lemma}\label{lem:centered}
For every $j$, the $3$-graph $H_j$ satisfies $\lambda(H_j)\le 2/9$. Moreover, for every probability vector $\mathbf{y}$ on $V(H_j)$,
\[
        P_{H_j}(\mathbf{y})\le \frac{2D}{3m_j} +\theta\left(\sum_{v\in V(H_j)}y_v^2-\frac1{m_j}\right).
\]
\end{lemma}

\begin{proof}
The inequality $\lambda(H_j)\le 2/9$ follows straightforwardly from the fact that $H_j$ is $3$-partite. It therefore suffices to prove the second assertion of Lemma~\ref{lem:centered}.

	Fix $j$, and write $\Gamma=\Gamma_j$, $S=S_j$, $n=|\Gamma|$, $m=3n$, and $X=X_j$, $Y=Y_j$, $Z=Z_j$. Let $\mathbf{y}=(y_v)_{v\in V(H_j)}$ be a probability vector. Denote its restrictions to $X,Y,Z$, under our identifications of these sets with $\Gamma$, by $\mathbf{a}=(a_x)_{x\in\Gamma}$, $\mathbf{b}=(b_y)_{y\in\Gamma}$ and $\mathbf{c}=(c_z)_{z\in\Gamma}$, respectively.
Define
\[
        A\coloneqq \sum_{x\in\Gamma}a_x, \qquad B\coloneqq \sum_{y\in\Gamma}b_y, \qquad C\coloneqq \sum_{z\in\Gamma}c_z,
\]
so that $A+B+C=1$. Also put
\[
        R_X\coloneqq \sum_{x\in\Gamma} a_x^2, \qquad R_Y\coloneqq \sum_{y\in\Gamma} b_y^2, \qquad R_Z\coloneqq \sum_{z\in\Gamma} c_z^2,
\]
and
\[
        \Delta_X\coloneqq R_X-\frac{A^2}{n}, \qquad \Delta_Y\coloneqq R_Y-\frac{B^2}{n}, \qquad \Delta_Z\coloneqq R_Z-\frac{C^2}{n}.
\]
The quantities $\Delta_X,\Delta_Y,\Delta_Z$ are non-negative by Cauchy's inequality.
Define 
\[
        T
        \coloneqq \sum_{x,y\in\Gamma}\sum_{s\in S}a_xb_yc_{xsy} .
\]
Note that $P_{H_j}(\mathbf{y})=6T$.

We use the standard terminology that the \emph{link} of a vertex $v$ in a $3$-graph is the graph whose edges are the pairs that form hyperedges together with $v$. For a bipartite graph with parts $U$ and $V$, its biadjacency matrix is the $U\times V$ matrix whose $(u,v)$-entry is $1$ if $uv$ is an edge, and is $0$ otherwise. 

\begin{claim}\label{clm:link-relabel}
For every vertex of $H_j$, its link between the other two vertex classes is a bipartite graph whose biadjacency matrix, after relabellings that are permutations of $\Gamma$, is identical to the adjacency matrix of $\Cay(\Gamma,S)$.
\end{claim}

\begin{proof}[Proof of the claim]
Fix $y_\ast\in Y$. The corresponding link has parts $X$ and $Z$. A pair $(x,z)$ is an edge precisely when $z=xsy_\ast$ for some $s\in S$. If we relabel the $X$-side by $\hat{x}\coloneqq x$ and the $Z$-side by $\hat{z}\coloneqq zy_\ast^{-1}$, then this condition becomes $\hat{z}=\hat{x}s$.

Fix $x_\ast\in X$. The corresponding link has parts $Y$ and $Z$. Relabel the $Y$-side by $\hat{y}\coloneqq y^{-1}$ and the $Z$-side by $\hat{z}\coloneqq z^{-1}x_\ast$. The relation $z=x_\ast sy$ is equivalent to $\hat{z}=\hat{y}s^{-1}$, which has the same right Cayley form since $S=S^{-1}$.

Finally, fix $z_\ast\in Z$. The corresponding link has parts $Y$ and $X$. Relabel the $Y$-side by $\hat{y}\coloneqq y^{-1}$ and the $X$-side by $\hat{x}\coloneqq z_\ast^{-1}x$. The relation $z_\ast=xsy$ is equivalent to $\hat{x}=\hat{y}s^{-1}$, again of the same right Cayley form. 

In each case the relabellings are permutations of $\Gamma$, and the resulting biadjacency relation is exactly the right Cayley relation.
\end{proof}

For fixed $y\in\Gamma$, define the vector $\mathbf{c}^{(y)}=(c^{(y)}_v)_{v\in\Gamma}$ by $c^{(y)}_v\coloneqq c_{vy}$. By this definition,
\[
        \sum_{x\in\Gamma}\sum_{s\in S}a_xc_{xsy} = \sum_{x\in\Gamma}\sum_{s\in S}a_xc^{(y)}_{xs}.
\]
The right-hand side is $\langle \mathbf{a},M\mathbf{c}^{(y)}\rangle$, where $M=M_{\Gamma,S}$ is the adjacency matrix of $\Cay(\Gamma,S)$. We view $\mathbf{a}$ and $\mathbf{c}^{(y)}$ as functions on $\Gamma$. The vector $\mathbf{a}$ has total mass $A$ and centred $L^2$ norm $\sqrt{\Delta_X}$. Also, since $v\mapsto vy$ is a permutation of $\Gamma$, the vector $\mathbf{c}^{(y)}$ has total mass $C$ and centred $L^2$ norm $\sqrt{\Delta_Z}$. Hence \eqref{eq:mixing} gives
\[
        \sum_{x\in\Gamma}\sum_{s\in S}a_xc_{xsy}
        = \langle \mathbf{a},M\mathbf{c}^{(y)}\rangle
        \le \frac{D}{n}AC+\theta\sqrt{\Delta_X\Delta_Z}.
\]
Multiplying this inequality by $b_y$ and summing over $y\in\Gamma$, using $\sum_{y\in\Gamma} b_y=B$, gives
\begin{equation}
        T\le \frac{D}{n}ABC+\theta B\sqrt{\Delta_X\Delta_Z}.
        \label{eq:XZ}
\end{equation}
The same argument applied to the links obtained by fixing a vertex in $X$ and in $Z$, using the relabellings from Claim~\ref{clm:link-relabel}, gives the analogous bounds
\begin{equation}
        T\le \frac{D}{n}ABC+\theta A\sqrt{\Delta_Y\Delta_Z},
        \label{eq:YZ}
\end{equation}
and
\begin{equation}
        T\le \frac{D}{n}ABC+\theta C\sqrt{\Delta_X\Delta_Y}.
        \label{eq:XY}
\end{equation}
Averaging \eqref{eq:XZ}, \eqref{eq:YZ} and \eqref{eq:XY}, and using the inequality $2\sqrt{uv}\le u+v$, we obtain
\begin{align*}
        T
        &\le \frac{D}{n}ABC
        +\frac{\theta}{3}
        \left(
        B\sqrt{\Delta_X\Delta_Z}
        +A\sqrt{\Delta_Y\Delta_Z}
        +C\sqrt{\Delta_X\Delta_Y}
        \right) 
        \le \frac{D}{n}ABC
        +\frac{\theta}{6}(\Delta_X+\Delta_Y+\Delta_Z).
\end{align*}
Multiplying by $6$, we get
\begin{equation}
        P_{H_j}(\mathbf{y})
        = 6T
        \le \frac{6D}{n}ABC +\theta(\Delta_X+\Delta_Y+\Delta_Z).
        \label{eq:lambdaH-main}
\end{equation}
Now $6ABC\le 2/9$, because $A+B+C=1$, and
\[
        \Delta_X+\Delta_Y+\Delta_Z =\sum_{v\in V(H_j)}y_v^2-\frac{A^2+B^2+C^2}{n} \le \sum_{v\in V(H_j)}y_v^2-\frac1{3n}.
\]
Since $m=3n$, \eqref{eq:lambdaH-main} implies
\[
        P_{H_j}(\mathbf{y}) \le \frac{2D}{9n} +\theta\left(\sum_{v\in V(H_j)}y_v^2-\frac1{3n}\right) = \frac{2D}{3m} +\theta\left(\sum_{v\in V(H_j)}y_v^2-\frac1m\right).
\]
This proves the centered estimate.
\end{proof}

\subsection{The multipartite outer construction}
\begin{construction}[The multipartite outer construction $Q_j$]\label{con:outer-Q}
For each $j$, take $k$ disjoint $m_j$-vertex sets $W_1,\ldots,W_k$. 
Inside each $W_i$, put one copy of the $3$-graph $H_j$ from Construction~\ref{con:spectral-gadget}. In addition, make every $3$-set $\{u,v,w\}$ an edge whenever its three vertices are not all contained in the same part $W_i$. Let the resulting $3$-graph be $Q_j$.
\end{construction}

Note that $|V(Q_j)|=km_j$, and
\[
        |Q_j|=k|H_j|+\binom{km_j}{3}-k\binom{m_j}{3}
        =kDn_j^2+\binom{km_j}{3}-k\binom{m_j}{3}.
\]

\begin{proposition}\label{prop:Q-lag}
Let $c_0\coloneqq \left(\frac{2D}{3}-3(k-1)\right)/{k^2} > 0$. 
As $j\to\infty$, we have
\[
        \lambda(Q_j) = \alpha+\frac{c_0+o(1)}{m_j}.
\]
\end{proposition}
Note that $c_0> 0$ is equivalent to $\frac{2D}{3}>3(k-1)$, which follows from the second inequality in \eqref{eq:k-choice}.

Note that the function $1-\sum_{i\in[k]} w_i^3$ on $\Delta_{k-1}$ has a unique maximum at the uniform vector. The next elementary lemma allows us to control the maximum value when a small continuous perturbation is added to the function.

\begin{lemma}\label{lem:perturb}
Let $k\ge 1$ be fixed. Let $f(\mathbf{w})\coloneqq 1-\sum_{i\in[k]} w_i^3$ for $\mathbf{w} = (w_1, \ldots, w_k) \in \Delta_{k-1}$. Let $h\colon\Delta_{k-1}\to\R$ be a continuous function. Let $\mathbf{w}_0=(1/k,\ldots,1/k)$. If $\tau\to 0^+$, then
\[
        \max_{\mathbf{w}\in\Delta_{k-1}}\{f(\mathbf{w})+\tau h(\mathbf{w})\} =f(\mathbf{w}_0)+\tau h(\mathbf{w}_0)+o(\tau).
\]
\end{lemma}

\begin{proof}
The function $f$ has a unique maximum on $\Delta_{k-1}$, attained at $\mathbf{w}_0$. The lower bound follows by evaluating at $\mathbf{w}_0$.

For the upper bound, let $\mathbf{w}(\tau)$ be a maximizer of $f+\tau h$. Since $h$ is bounded and
\[
        f(\mathbf{w}(\tau))+\tau h(\mathbf{w}(\tau))\ge f(\mathbf{w}_0)+\tau h(\mathbf{w}_0),
\]
we have $f(\mathbf{w}(\tau))\ge f(\mathbf{w}_0)-O(\tau)$. Every limit point of $\mathbf{w}(\tau)$ as $\tau\to 0^+$ is therefore a maximizer of $f$, hence is $\mathbf{w}_0$. Thus $\mathbf{w}(\tau)\to \mathbf{w}_0$. By continuity, $h(\mathbf{w}(\tau))=h(\mathbf{w}_0)+o(1)$, and since $f(\mathbf{w}(\tau))\le f(\mathbf{w}_0)$, we get
\[
        f(\mathbf{w}(\tau))+\tau h(\mathbf{w}(\tau))\le f(\mathbf{w}_0)+\tau h(\mathbf{w}_0)+o(\tau). \qedhere
\]
\end{proof}

\bigskip 

We now present the proof of Proposition~\ref{prop:Q-lag}. 

\begin{proof}[Proof of Proposition~\ref{prop:Q-lag}]
Let $\mathbf{z}$ be a probability vector on $V(Q_j)$. For $i\in [k]$, define
\[
        w_i\coloneqq \sum_{v\in W_i}z_v,\qquad s_i\coloneqq \sum_{v\in W_i}z_v^2,
\]
so that $\sum_{i\in[k]} w_i=1$. If $w_i>0$, let $\mathbf{y}_i=(y_{i,v})_{v\in W_i}$ be the normalized restriction of $\mathbf{z}$ to $W_i$, namely $y_{i,v}\coloneqq z_v/w_i$, and put
\[
        r_i\coloneqq \sum_{v\in W_i}y_{i,v}^2=\frac{s_i}{w_i^2}.
\]
If $w_i=0$, then $z_v=0$ for every $v\in W_i$; in this case, as a harmless convention, we take $\mathbf{y}_i$ to be the uniform probability vector on $W_i$ and put $r_i\coloneqq 1/m_j$. With this convention, $s_i=w_i^2r_i$ for every $i$.

The contribution to $P_{Q_j}(\mathbf{z})$ from all triples not contained in a single $W_i$ is computed as follows. We first count all ordered triples not contained in one part, allowing repetitions, and then subtract the ordered triples with exactly two equal vertices and the third vertex in a different part:
\begin{align*}
        \sum_{\substack{(u,v,w)\in V(Q_j)^3\\
        \{u,v,w\}\not\subseteq W_i\text{ for all }i\in[k]}}
        z_uz_vz_w
        -
        3\sum_{i\in[k]}\sum_{\substack{u\in W_i\\ v\notin W_i}}z_u^2z_v 
        & =
        \Big(1-\sum_{i\in[k]}\Big(\sum_{u\in W_i}z_u\Big)^3\Big)
        -
        3\sum_{i\in[k]}\Big(\sum_{u\in W_i}z_u^2\Big)
        \Big(\sum_{v\notin W_i}z_v\Big)  \\
        & =
        1-\sum_{i\in[k]} w_i^3-3\sum_{i\in[k]}(1-w_i)s_i .
\end{align*}
Therefore
\begin{equation}
        P_{Q_j}(\mathbf{z}) =1-\sum_{i\in[k]} w_i^3-3\sum_{i\in[k]}(1-w_i)w_i^2r_i +\sum_{i\in[k]} w_i^3P_{H_j}(\mathbf{y}_i).
        \label{eq:Q-exact}
\end{equation}
Using Lemma~\ref{lem:centered}, $P_{H_j}(\mathbf{y}_i)\le \frac{2D}{3m_j}+\theta\left(r_i-\frac1{m_j}\right)$. Substituting this into \eqref{eq:Q-exact}, we get
\begin{equation}
        P_{Q_j}(\mathbf{z}) \le 1-\sum_{i\in[k]} w_i^3 +\left(\frac{2D}{3m_j}-\frac{\theta}{m_j}\right)\sum_{i\in[k]} w_i^3 +\sum_{i\in[k]} w_i^2r_i\bigl((\theta+3)w_i-3\bigr).
        \label{eq:Q-upper1}
\end{equation}

We first show that, for every probability vector $\mathbf{z}$ maximizing $P_{Q_j}$, all part masses $w_i$ are small. Since $\lambda(H_j)\le 2/9$, equation \eqref{eq:Q-exact} gives
\[
        P_{Q_j}(\mathbf{z})\le 1-\frac79\sum_{i\in[k]} w_i^3.
\]
On the other hand, let $\mathbf{u}$ be the uniform vector on $V(Q_j)$ and evaluate $P_{Q_j}$ at $\mathbf{u}$. 
Then \eqref{eq:Q-exact} gives
\begin{align}\label{eq:Q-lower}
    \lambda(Q_j)\ge P_{Q_j}(\mathbf{u})
        &=
        1-k\cdot\frac1{k^3}
        -3k\left(1-\frac1k\right)\frac1{k^2}\frac1{m_j}
        +k\cdot\frac1{k^3}\frac{2D}{3m_j} \notag \\
        &=1-\frac1{k^2}-\frac{3(k-1)}{k^2m_j}
        +\frac{2D}{3k^2m_j} 
        =\alpha+\frac{c_0}{m_j}.
\end{align}
If $\mathbf{z}$ is a maximizing vector, then the last two displayed inequalities imply
\[
        1-\frac1{k^2}+\frac{c_0}{m_j}
        \le \lambda(Q_j)=P_{Q_j}(\mathbf{z})
        \le 1-\frac79\sum_{i\in[k]} w_i^3 .
\]
Hence, since $c_0>0$, we have 
\[
        \max_{i\in[k]} w_i^3
        \le\sum_{i\in[k]} w_i^3
        \le \frac97\left(\frac1{k^2}-\frac{c_0}{m_j}\right)
        \le \frac{9}{7k^2}, 
\]
which implies that $\max_{i\in[k]} w_i \le \left(9/7\right)^{1/3}k^{-2/3}$. 
Combining this with the first inequality in \eqref{eq:k-choice}, we obtain that every maximizing vector satisfies $w_i<\frac3{\theta+3}$ for every $i$. Hence $(\theta+3)w_i-3<0$ for every $i$. Since $r_i\ge 1/m_j$ by Cauchy's inequality, the last term in \eqref{eq:Q-upper1}, for a maximizing vector, is at most its value when each $r_i$ is replaced by $1/m_j$. Consequently
\begin{equation}
        \lambda(Q_j)\le \max_{\mathbf{w} \in \Delta_{k-1}} \left[ 1-\sum_{i\in[k]} w_i^3 +\frac1{m_j} \left( \left(\frac{2D}{3}+3\right)\sum_{i\in[k]} w_i^3-3\sum_{i\in[k]} w_i^2 \right) \right].
        \label{eq:Q-upper2}
\end{equation}
Applying Lemma~\ref{lem:perturb} with $h(\mathbf{w}) \coloneqq \left(\frac{2D}{3}+3\right)\sum_{i\in[k]} w_i^3-3\sum_{i\in[k]} w_i^2$ and $\tau = 1/m_j$ gives
\begin{align}
        \lambda(Q_j)
        &\le
        \alpha
        +\frac1{m_j}
        \left(
        \frac{\frac{2D}{3}+3}{k^2}-\frac3k
        \right)
        +o(m_j^{-1})   
        =\alpha
        +\frac{c_0}{m_j}
        +o(m_j^{-1}).
        \label{eq:Q-upper3}
\end{align}
Combining  \eqref{eq:Q-lower} and \eqref{eq:Q-upper3} proves the asymptotic formula.
\end{proof}

\subsection{Proof of Theorem~\ref{thm:smooth-sequence}}
We complete the proof of Theorem~\ref{thm:smooth-sequence} in this subsection. 
Recall that Proposition~\ref{prop:Q-lag} gives the right asymptotic size of the gaps $\lambda(Q_j)-\alpha$, but it does not by itself make them monotone. We use the following elementary extraction observation to pass to a decreasing subsequence without losing the smooth growth of the parameters.

\begin{lemma}\label{lem:extraction}
Let $(m_j)_{j=1}^\infty$ be an increasing sequence of positive integers with $m_j\to\infty$ and $m_{j+1}/m_j\to 1$. Suppose that a sequence $(a_j)_{j=1}^\infty$ satisfies $a_j=(c+o(1))/{m_j}$ for some $c>0$. Then there is a subsequence $j_1<j_2<j_3<\cdots$ such that $a_{j_1}>a_{j_2}>a_{j_3}>\cdots$ and ${m_{j_{i+1}}}/{m_{j_i}}\to 1$. 
Consequently, ${a_{j_i}}/{a_{j_{i+1}}}\to 1$.
\end{lemma}

\begin{proof}
Write $a_j=(c+\eps_j)/{m_j}$, where $\eps_j\to 0$. After discarding finitely many terms, we may assume $c+\eps_j>0$ for all $j$. We construct the subsequence inductively starting with $j_1:=1$. Suppose $j_i$ has been chosen. Let
\[
        \eta_i \coloneqq \max\left\{ \sup_{j\ge j_i}|\eps_j|,\ \sup_{j\ge j_i}\left|\frac{m_{j+1}}{m_j}-1\right|,\ \frac{1}{i^2} \right\}.
\]

Clearly, $\eta_i<\infty$. Put $\vartheta_i\coloneqq \sqrt{\eta_i}$. Let $j_{i+1}>j_i$ be the first index such that $m_{j_{i+1}}\ge (1+\vartheta_i)m_{j_i}$. This index exists because $m_j\to\infty$. 

Let us analyse the obtained sequence $(j_i)_{i\ge 1}$.
Since $j_i\ge i$, we have $j_i\to\infty$. Thus
\[
        \sup_{j\ge j_i}|\eps_j|\to 0
        \qquad\text{and}\qquad
        \sup_{j\ge j_i}\left|\frac{m_{j+1}}{m_j}-1\right|\to 0,
\]
by $\eps_j\to0$ and $m_{j+1}/m_j\to1$, respectively. Hence $\eta_i\to 0$.

By the minimality of $j_{i+1}$, we have $m_{j_{i+1}-1}<(1+\vartheta_i)m_{j_i}$. Also, by the definition of $\eta_i$, we have $m_{j_{i+1}} \le (1+\eta_i)m_{j_{i+1}-1}$. Hence
\[
        1+\vartheta_i \le \frac{m_{j_{i+1}}}{m_{j_i}} \le (1+\vartheta_i)(1+\eta_i),
\]
and therefore $m_{j_{i+1}}/m_{j_i}\to 1$.

It remains to prove that the chosen terms are decreasing. Since both $j_i$ and $j_{i+1}$ are at least $j_i$, we have $|\eps_{j_i}|\le \eta_i$ and $|\eps_{j_{i+1}}|\le \eta_i$. Therefore
\[
        a_{j_i}\ge \frac{c-\eta_i}{m_{j_i}}, \qquad\text{and}\qquad a_{j_{i+1}}\le \frac{c+\eta_i}{m_{j_{i+1}}} \le \frac{c+\eta_i}{(1+\sqrt{\eta_i})m_{j_i}}.
\]
For all sufficiently large $i$, we have $c-\eta_i>\frac{c+\eta_i}{1+\sqrt{\eta_i}}$, because this inequality is equivalent to $c\sqrt{\eta_i}>2\eta_i+\eta_i^{3/2}$, which holds when $\eta_i$ is small enough. Hence $a_{j_i}>a_{j_{i+1}}$ for all sufficiently large $i$. Discarding finitely many initial terms, we obtain a strictly decreasing subsequence.

Finally,
\[
        \frac{a_{j_i}}{a_{j_{i+1}}} = \frac{(c+\eps_{j_i})/m_{j_i}}{(c+\eps_{j_{i+1}})/m_{j_{i+1}}} = \frac{m_{j_{i+1}}}{m_{j_i}}(1+o(1))\to 1, 
\]
as desired. 
\end{proof}

Apply Lemma~\ref{lem:extraction} to $a_j\coloneqq \lambda(Q_j)-\alpha$. By Proposition~\ref{prop:Q-lag}, $a_j=(c_0+o(1))/{m_j}$ with $c_0 > 0$. Thus there is a subsequence $j_1<j_2<\cdots$ such that $\lambda(Q_{j_1})>\lambda(Q_{j_2})>\lambda(Q_{j_3})>\cdots$ and ${m_{j_{i+1}}}/{m_{j_i}}\to 1$.

Define $P_i\coloneqq Q_{j_i}$. Then $\lambda(P_i)\downarrow \alpha$, and
\[
        \frac{\lambda(P_i)-\alpha}{\lambda(P_{i+1})-\alpha} = \frac{(c_0+o(1))/m_{j_i}}{(c_0+o(1))/m_{j_{i+1}}} = \frac{m_{j_{i+1}}}{m_{j_i}}(1+o(1))\to 1,
\]
which proves Theorem~\ref{thm:smooth-sequence}.

\section{Filling an interval by the slowly convergent sequence}
\label{sec:recursive-interval}
In this section, we prove Theorem~\ref{thm:interval-3} below, which will be the main stepping stone for proving Theorem~\ref{thm:main}.

\begin{theorem}\label{thm:interval-3} 
There exist reals $\mu< \nu_1$ such that $[\mu,\nu_1]\subseteq \Pi^{(3)}_\infty$.
\end{theorem}


Informally speaking, our proof strategy is to consider blowups of $K_m^{(3)}$ for some fixed large $m$ where we apply recursion inside the first part and, at each recursion level $i$, can insert the $3$-graph $P_{a_i}$ returned by Theorem~\ref{thm:smooth-sequence} into the second part of the blowup. Each choice of $a_1<a_2<a_3<\dots$ gives some element of $\Pi^{(3)}_\infty$ in the limit as the number of vertices in the construction tends to infinity. If we decrease some $a_i$ then this value increases. So, for a given real $y$ in the interval that we aim to cover, we choose $a_i$ inductively for $i=1,2,3,\dots$, with $a_{i}$ being the smallest index above $a_{i-1}$ such that the final value can stay below~$y$. In the other direction, the smoothness of the sequence $(\lambda(P_i))_{i\ge 1}$ will guarantee that these values necessarily approach $y$.

Let us provide all formal details. Let $(P_i)_{i\ge 1}$ be the sequence of $3$-graphs returned by Theorem~\ref{thm:smooth-sequence}.
Recall that $\alpha=1 - 1/k^2$ is the limit of $\lambda(P_i)$.
Choose an integer $m\ge 4$ such that
\begin{equation}
        \lambda(K_{m-1}^{(3)})=\frac{(m-2)(m-3)}{(m-1)^2}>\alpha .
        \label{eq:m-choice}
\end{equation}
This is possible because $\lambda(K_s^{(3)})=(s-1)(s-2)/s^2\to 1$ as $s\to\infty$. In particular, straightforward calculations show that $m\ge k$.
Write
\[
        H_i\coloneqq P_i, \qquad \beta_i\coloneqq \lambda(H_i), \qquad \eps_i\coloneqq \beta_i-\alpha .
\]
Thus
\begin{equation}
        \eps_i>0, \qquad \eps_i\downarrow 0, \qquad\text{and}\qquad \frac{\eps_i}{\eps_{i+1}}\to 1.
        \label{eq:eps-slow}
\end{equation}

\begin{construction}[$A$-configurations]\label{con:A-config} 
Let $A=\{a_1<a_2<\cdots\}\subseteq\N$ be an infinite sequence. An $A$-configuration is obtained recursively as follows. Partition the current vertex set into $m$ parts; add all triples meeting each part in at most one vertex; put a blow-up of $H_{a_1}$ inside the second part; and continue recursively inside the first part with the new infinite sequence $A\setminus\{a_1\}$. 
\end{construction}

We next encode the limiting densities of these recursive configurations by one-variable maps. The first part carries the recursive tail, the second part carries one of the graphs $H_i$, and the remaining parts contribute only through the complete cross-part triples.

For every $\mathbf{z} = (z_1, \ldots, z_m) \in \mathbb{R}^{m}$, let
\[
        g_m(\mathbf{z})
        \coloneqq P_{K_m^{(3)}}(\mathbf{z})
        = 6\sum_{\{a,b,c\}\in\binom{[m]}{3}}z_az_bz_c .
\]
For $b,t\in[0,1]$, define
\begin{equation}
        \psi_b(t)\coloneqq \max_{\mathbf{z}\in\Delta_{m-1}} \bigl\{ g_m(\mathbf{z})+t z_1^3+bz_2^3\bigr\}.
        \label{eq:psi-def}
\end{equation}
\begin{fact}\label{fact:psi-basic}
For each fixed $b\in[0,1]$, the map $t\mapsto \psi_b(t)$ is non-decreasing and $1$-Lipschitz. For each fixed $t\in[0,1]$, the map $b\mapsto \psi_b(t)$ is also non-decreasing and $1$-Lipschitz. That is,
\[
        |\psi_b(t)-\psi_b(t')|\le |t-t'| \qquad\text{and}\qquad
        |\psi_b(t)-\psi_{b'}(t)|\le |b-b'|
\]
for all $b,b',t,t'\in[0,1]$.
Moreover, $\psi_b(t)\ge t$ for every $t\in[0,1]$.
\end{fact}

\begin{proof}
For fixed $\mathbf{z}\in\Delta_{m-1}$, changing $t$ to $t'$ changes the expression inside the maximum by $(t-t')z_1^3$, whose absolute value is at most $|t-t'|$ because $0\le z_1^3\le 1$. Taking the maximum over $\mathbf{z}$ gives the $1$-Lipschitz bound. The same observation also shows monotonicity in $t$. The proof for the variable $b$ is identical, using $(b-b')z_2^3$ instead. Finally, taking $\mathbf{z}=(1,0,\ldots,0)$ gives $\psi_b(t)\ge t$.
\end{proof}

We next define the base value of the constant-$\alpha$ recursion. For $a\in[0,1]$, put
\begin{equation}
        \mu(a)\coloneqq
        \max_{\substack{\mathbf{z}\in\Delta_{m-1}\\ z_1<1}}
        \frac{g_m(\mathbf{z})+az_2^3}{1-z_1^3}.
        \label{eq:mu-a-def}
\end{equation}
Recall that $\alpha=1-1/k^2$ and let 
\[
        \mu\coloneqq \mu(\alpha).
\]

Let us informally explain the meaning of $\mu$. Consider the recursive pattern based on $K_m^{(3)}$: the first part carries another copy of the same recursive pattern, while the second part carries a non-recursive pattern of Lagrangian $\alpha$. Let $\mu$ be the maximum Lagrangian that we can achieve. It satisfies
\begin{equation}
        g_m(\mathbf{z})+\mu z_1^3+\alpha z_2^3\le \mu \qquad\text{for all } \mathbf{z}\in\Delta_{m-1}.
        \label{eq:mu-ineq}
\end{equation}
Indeed, if $\mathbf{z}$ records the part ratios on the top level of a recursive construction $G$ with $n\to\infty$ vertices and we make optimal choices within the first and second parts then, up to an additive error term $o(1)$,  $\varrho(G)$ is the sum of $g_m(\mathbf{z})$ (the Lagrangian polynomial of $K_m^{(3)}$), $\mu z_1^3$ and $\alpha z_2^3$ (the contributions from the optimal constructions within the first and second parts respectively). On the other hand, the edge density of $G$ is at most $\mu+o(1)$ by the maximality of $\mu$. This explains the inequality in~\eqref{eq:mu-ineq}.  As we will need to refer to this inequality later, let us remark that an easy formal calculation shows that our definition of $\mu:=\mu(\alpha)$ via~\eqref{eq:mu-a-def} is equivalent to it being the smallest real satisfying~\eqref{eq:mu-ineq}.

Let $\mathcal X$ be the set of maximizers in \eqref{eq:mu-a-def} with $a=\alpha$, that is, 
\begin{align*}
    \mathcal X
    \coloneqq 
    \left\{\mathbf{z} = (z_1, \ldots, z_m) \in \Delta_{m-1} \colon z_1 < 1,\ \frac{g_m(\mathbf{z}) + \alpha z_2^3}{1-z_1^3} = \mu \right\}. 
\end{align*}

\begin{lemma}\label{lem:mu-basic}
The function $a\mapsto \mu(a)$ is continuous on $[0,1]$ and satisfies $0\le \mu(a)\le 1$. Moreover,
\begin{equation}
        \mu\ge \lambda(K_{m-1}^{(3)})>\alpha .
        \label{eq:mu-greater-alpha}
\end{equation}
Let $\mathbf{e}_1 \coloneqq (1,0,\ldots,0)$. The set $\mathcal X$ is non-empty and compact, and there is a constant $\rho>0$ such that
\[
        \|\mathbf{z}-\mathbf{e}_1\|_2\ge \rho \qquad\text{for all } \mathbf{z}\in\mathcal X.
\]
If $\mathbf{z}^*\in\mathcal X$ is chosen so that $(z_2^*)^3$ is maximal
and we let
\begin{equation}
        q\coloneqq (z_1^*)^3\qquad\text{and}\qquad  R\coloneqq (z_2^*)^3,
        \label{eq:qR-def}
\end{equation}
then
\begin{equation}
        0<q<1 \qquad\text{and}\qquad R>0.
        \label{eq:qR-positive}
\end{equation}
\end{lemma}

\begin{proof}
The quotient in \eqref{eq:mu-a-def} extends continuously to $\mathbf{z}=\mathbf{e}_1$ by assigning value $0$: indeed, if $1-z_1=\delta$, then $g_m(\mathbf{z})=O(\delta^2)$ and $z_2^3=O(\delta^3)$, while $1-z_1^3\ge \delta$. Moreover, $g_m(\mathbf{z})+a z_2^3\le 1-z_1^3$ for $a\in[0,1]$, since the right-hand side is the ordered weight of all triples of parts not entirely in part $1$, allowing repetitions. Hence $0\le \mu(a)\le 1$. Since the extended quotient is continuous in $(a,\mathbf{z})$ on the compact set $[0,1]\times\Delta_{m-1}$, its maximum over $\mathbf{z}\in\Delta_{m-1}$ is continuous in $a$. 
By \eqref{eq:m-choice}, taking $z_1=0$ and distributing the mass equally among the coordinates $2,\ldots,m$ gives \eqref{eq:mu-greater-alpha}. 

The set $\mathcal X$ is non-empty and compact by the continuous extension described above. Since the quotient tends to $0$ at $\mathbf{e}_1$, whereas $\mu>0$ by \eqref{eq:mu-greater-alpha}, there is a constant $\rho>0$ such that $\|\mathbf{z}-\mathbf{e}_1\|_2\ge\rho$ for every $\mathbf{z}\in\mathcal X$. Thus $q<1$. If $z^*_1=0$, then either some coordinate $z^*_h>0$ with $h\ge 3$, in which case swapping coordinates $1$ and $h$ leaves $g_m$ and $z_2$ unchanged but increases the left-hand side of \eqref{eq:mu-ineq} by $\mu (z_h^*)^3>0$, or all mass lies on coordinate $2$, in which case \eqref{eq:mu-a-def} gives $\mu=\alpha$, contradicting \eqref{eq:mu-greater-alpha}. Thus $q>0$. Similarly, if $z^*_2=0$, then either some coordinate $h\ge 3$ has positive mass, or $\mathbf{z}^*=\mathbf{e}_1$. The second possibility cannot occur because $\mathbf{z}^*\in\mathcal X$ and every point of $\mathcal X$ has distance at least $\rho$ from $\mathbf{e}_1$. In the first case, swapping coordinates $2$ and $h$ leaves $g_m$ and $z_1$ unchanged but increases the left-hand side of \eqref{eq:mu-ineq} by $\alpha (z_h^*)^3>0$, a contradiction. Hence $R>0$.
\end{proof}

For the rest of this section, fix $\mathbf{z}^*\in\mathcal X$ and the corresponding constants $q$ and $R$ as in Lemma~\ref{lem:mu-basic}.

We shall need the following one-step estimate.

\begin{lemma}\label{lem:one-step} As $\eta\to 0^+$, we have $\psi_{\alpha+\eta}(\mu)\le \mu+R\eta+o(\eta)$. \end{lemma}

\begin{proof}
Let $\mathbf{z}^{(\eta)}=(z_1^{(\eta)},\ldots,z_m^{(\eta)})$ be a maximizer in \eqref{eq:psi-def} with $(b,t)=(\alpha+\eta,\mu)$. Then
\begin{equation}
        \psi_{\alpha+\eta}(\mu)-\mu = \bigl(g_m(\mathbf{z}^{(\eta)})+\mu (z_1^{(\eta)})^3+\alpha (z_2^{(\eta)})^3-\mu\bigr) +\eta (z_2^{(\eta)})^3.
        \label{eq:one-step-expand}
\end{equation}
The bracketed term is non-positive by \eqref{eq:mu-ineq}. Also, by using the fixed vector $\mathbf{z}^*$ as a test vector in \eqref{eq:psi-def}, we have
\begin{equation}
        \psi_{\alpha+\eta}(\mu)-\mu \ge g_m(\mathbf{z}^*)+\mu (z_1^*)^3+(\alpha+\eta)(z_2^*)^3-\mu =R\eta.
        \label{eq:one-step-lower}
\end{equation}
Suppose that the asserted upper bound fails. Then there are $\delta>0$ and a sequence $\eta_n\to 0^+$ such that
\[
        \psi_{\alpha+\eta_n}(\mu)-\mu>(R+\delta)\eta_n.
\]
Put $\mathbf{z}^{(n)}\coloneqq \mathbf{z}^{(\eta_n)}=(z_1^{(n)},\ldots,z_m^{(n)})$. By \eqref{eq:one-step-expand}, this forces $(z_2^{(n)})^3>R+\delta$ for all large $n$. Moreover, the bracketed term in \eqref{eq:one-step-expand} tends to $0$: it is always non-positive, while \eqref{eq:one-step-lower} and $(z_2^{(n)})^3\le 1$ give
\[
        0\le \mu-g_m(\mathbf{z}^{(n)})-\mu (z_1^{(n)})^3-\alpha (z_2^{(n)})^3 \le \eta_n (z_2^{(n)})^3-R\eta_n \le \eta_n.
\]
Since $\Delta_{m-1}$ is compact, after passing to a subsequence we may assume that $\mathbf{z}^{(n)}$ converges to some point $\mathbf{z}^0\in\Delta_{m-1}$. Then equality holds in \eqref{eq:mu-ineq} at $\mathbf{z}^0$, and $(z_2^0)^3\ge R+\delta$. In particular, $\mathbf{z}^0\ne\mathbf{e}_1$, so $\mathbf{z}^0\in\mathcal X$, contradicting the choice of $R=\max_{\mathbf{z}\in\mathcal X}z_2^3$.
\end{proof}

For an infinite sequence $b=(b_1,b_2,\ldots)$ with $b_i\in[0,1]$, the Lagrangian value of the $s$-level truncation of the recursive pattern is
\[
        \Lambda_s(b)\coloneqq \psi_{b_1}\circ\psi_{b_2}\circ\cdots\circ\psi_{b_s}(0),
\]
where the maps are composed from right to left.  Equivalently, for fixed $s$, set $t_s\coloneqq 0$ and define recursively
\[
        t_{\ell-1}\coloneqq \psi_{b_\ell}(t_\ell) \qquad \text{for } \ell=s,s-1,\ldots,1.
\]
Then $\Lambda_s(b)=t_0$. 
Define the infinite recursive value by
\begin{equation}
        \Lambda(b)\coloneqq \lim_{s\to\infty}\Lambda_s(b).
        \label{eq:Lambda-def}
\end{equation}
The limit exists because inserting one more level can only increase the value, while all values are at most $1$. Indeed, since $b,t\le 1$ and $\sum_{i\in[m]} z_i=1$, for all $b,t\in[0,1]$ and $\mathbf{z}\in\Delta_{m-1}$ we have 
\[
        g_m(\mathbf{z})+tz_1^3+bz_2^3
        \le g_m(\mathbf{z})+z_1^3+z_2^3
        \le \Big(\sum_{i\in [m]} z_i\Big)^3=1.
\]

\begin{lemma}\label{lem:alpha-tail-value}
The constant-$\alpha$ recursive value is $\mu$, that is, $\Lambda(\alpha,\alpha,\alpha,\ldots)=\mu$. 
\end{lemma}

\begin{proof}
The inequality \eqref{eq:mu-ineq} gives $\psi_\alpha(t)\le \mu$ whenever $t\le \mu$, while evaluating at $\mathbf{z}^*$ gives
\begin{equation}
        \psi_\alpha(t) \ge g_m(\mathbf{z}^*)+tq+\alpha R =\mu-q(\mu-t).
        \label{eq:alpha-contraction}
\end{equation}
Therefore the iterates $t_{s+1}\coloneqq \psi_\alpha(t_s)$, $t_0=0$, satisfy $0\le \mu-t_{s+1}\le q(\mu-t_s)$ and hence $t_s\to \mu$ (since $q<1$ by~\eqref{eq:qR-positive}).
\end{proof}

We now define the full-tail values of the actual sequence $(\beta_i)_{i \ge 1}$. For $j\ge 1$, set
\begin{equation}
        \nu_j\coloneqq \Lambda(\beta_j,\beta_{j+1},\beta_{j+2},\ldots).
        \label{eq:nuj-def}
\end{equation}
Thus $\nu_j$ is the value of the recursive tail that starts with $H_j$ and then uses all later graphs.

\begin{lemma}\label{lem:nuj-properties}
The full-tail values satisfy
\begin{equation}
        \nu_j=\psi_{\beta_j}(\nu_{j+1}).
        \label{eq:nuj-recursion}
\end{equation}
Moreover,
\begin{equation}
        \nu_j\to \mu.
        \label{eq:nuj-to-mu}
\end{equation}
After possibly discarding finitely many initial terms and relabelling the sequence, we also have
\begin{equation}
        \psi_{\beta_j}(\mu)\le \nu_{j+1} \qquad\text{for every }j\ge 1.
        \label{eq:overlap-basic}
\end{equation}
\end{lemma}

\begin{proof}
The recursion \eqref{eq:nuj-recursion} follows directly from the definition of $\nu_j$ and the definition of $\Lambda$. To prove \eqref{eq:nuj-to-mu}, note first that $\nu_j\ge \mu$, by Lemma~\ref{lem:alpha-tail-value}, monotonicity in all side densities, and the inequalities $\beta_i \ge \alpha$ for all $i\ge1$. Since $\beta_i\le \beta_j$ for all $i\ge j$, monotonicity of $\psi_b(t)$ in both variables gives
\[
        \nu_j=\Lambda(\beta_j,\beta_{j+1},\ldots) \le \Lambda(\beta_j,\beta_j, \ldots).
\]
The defining inequality for $\mu(\beta_j)$, namely
\[
        g_m(\mathbf{z})+\mu(\beta_j)z_1^3+\beta_jz_2^3\le \mu(\beta_j) \qquad\text{for all } \mathbf{z}\in\Delta_{m-1},
\]
shows inductively that all finite iterates with the constant sequence $(\beta_j,\beta_j,\ldots)$, started from $0$, are at most $\mu(\beta_j)$. Hence $\nu_j\le \mu(\beta_j)=\mu(\alpha+\eps_j)$. The continuity of $\mu(a)$ and $\eps_j\to 0$ prove \eqref{eq:nuj-to-mu}.

It remains to prove \eqref{eq:overlap-basic}. Set $\delta_j\coloneqq \nu_j-\mu$. Evaluating \eqref{eq:nuj-recursion} at the fixed vector $\mathbf{z}^*$ gives
\begin{align*}
        \nu_j
        =\psi_{\beta_j}(\nu_{j+1})
        &\ge g_m(\mathbf{z}^*)+\nu_{j+1}(z_1^*)^3+\beta_j(z_2^*)^3  \\
        &=g_m(\mathbf{z}^*)+(\mu+\delta_{j+1})q+(\alpha+\eps_j)R 
        =\mu+q\delta_{j+1}+R\eps_j,
\end{align*}
where the last equality uses $g_m(\mathbf{z}^*)+\mu q+\alpha R=\mu$. Hence
\begin{equation}
        \delta_j\ge q\delta_{j+1}+R\eps_j.
        \label{eq:delta-recursion}
\end{equation}
Consequently, for every fixed $N\ge 1$,
\begin{equation}
        \delta_{j+1} \ge R\sum_{s\in[N]} q^{s-1}\eps_{j+s} \qquad\text{for every }j.
        \label{eq:delta-sum}
\end{equation}
Put $\sigma\coloneqq q>0$. Taking $N=2$ in \eqref{eq:delta-sum}, and using \eqref{eq:eps-slow}, we have $\eps_{j+1}/\eps_j\to 1$ and $\eps_{j+2}/\eps_j\to 1$. Hence, for all sufficiently large $j$,
\[
        \eps_{j+1}+q\eps_{j+2}
        \ge \left(1+q-\frac{\sigma}{2}\right)\eps_j
        =\left(1+\frac{\sigma}{2}\right)\eps_j.
\]
Thus \eqref{eq:delta-sum} implies that, for all sufficiently large $j$,
\begin{equation}
        \nu_{j+1}-\mu\ge R(1+\sigma/2)\eps_j.
        \label{eq:nu-lower-gap}
\end{equation}
On the other hand, Lemma~\ref{lem:one-step} gives, for all sufficiently large $j$,
\begin{equation}
        \psi_{\beta_j}(\mu)-\mu =\psi_{\alpha+\eps_j}(\mu)-\mu \le R(1+\sigma/4)\eps_j.
        \label{eq:psi-upper-gap}
\end{equation}
The estimates \eqref{eq:nu-lower-gap} and \eqref{eq:psi-upper-gap} have been proved only for all sufficiently large $j$. We discard the finite initial segment for which they may fail and relabel the remaining sequence so that the first remaining index is again called $1$. Then \eqref{eq:nu-lower-gap} and \eqref{eq:psi-upper-gap} give \eqref{eq:overlap-basic} for every $j\ge 1$.
\end{proof}

We now translate this into the set-valued recursive construction.

Let $T_j\coloneqq \{j,j+1,j+2,\ldots\}$. We use the following notation.

\begin{itemize}
\item If $A=\{a_1<a_2<\cdots\}\subseteq\N$ is infinite, define
\[
        \Lambda_A\coloneqq \Lambda(\beta_{a_1},\beta_{a_2},\ldots).
\]

\item If $A=\{a_1<\cdots<a_s\}\subseteq\N$ is finite, define
\begin{equation}
        \Lambda_A\coloneqq \psi_{\beta_{a_1}}\circ\cdots\circ\psi_{\beta_{a_s}}(\mu).
        \label{eq:Lambda-finite}
\end{equation}
Equivalently, by \eqref{eq:nuj-to-mu} and the $1$-Lipschitz property in the recursive variable, (in the finite case) $\Lambda_A$ is the limit of $\Lambda_{A\cup T_j}$ as $j\to\infty$.
\end{itemize}

For a finite set $C=\{c_1<\cdots<c_s\}\subseteq\N$, define
\[
        \Psi_C\coloneqq \psi_{\beta_{c_1}}\circ\cdots\circ\psi_{\beta_{c_s}},
\]
with $\Psi_\emptyset$ equal to the identity map.
Thus the finite definition above can be written as $\Lambda_C=\Psi_C(\mu)$; more generally, if a tail has recursive value $t$, then attaching the finite prefix $C$ gives the value $\Psi_C(t)$.

\begin{lemma}\label{lem:overlap-property}
Let $j\ge 1$ and let $C\subseteq [j-1]$. Then
\begin{equation}
        \Lambda_{C\cup\{j\}} =\Psi_C(\psi_{\beta_j}(\mu)) \le \Psi_C(\nu_{j+1}) =\Lambda_{C\cup T_{j+1}}.
        \label{eq:overlap-property}
\end{equation}
\end{lemma}

\begin{proof}
Since each map $\psi_b$ is non-decreasing in the recursive variable, so is $\Psi_C$. Therefore \eqref{eq:overlap-basic} gives
\[
        \Psi_C(\psi_{\beta_j}(\mu))\le \Psi_C(\nu_{j+1}).
\]
The identity $\Lambda_{C\cup\{j\}}=\Psi_C(\psi_{\beta_j}(\mu))$ follows directly from the finite definition \eqref{eq:Lambda-finite}, since all elements of $C$ are smaller than $j$.

It remains to identify the right-hand side. For $N\ge j+1$, the finite truncation with tail $\{j+1,\ldots,N\}$ has value $\Psi_C\bigl(\psi_{\beta_{j+1}}\circ\cdots\circ\psi_{\beta_N}(0)\bigr)$. 
As $N\to\infty$, the inner term converges to $\nu_{j+1}$ by the definition of $\nu_{j+1}$, and the finite composition $\Psi_C$ is continuous because it is $1$-Lipschitz. Hence this value tends to $\Psi_C(\nu_{j+1})$, which is precisely $\Lambda_{C\cup T_{j+1}}$.
\end{proof}

We shall use the following special case of the infinite $A$-configuration argument from~\cite{Pikhurko2014}.

\begin{lemma}[{\cite[Proof of Theorem~2, p.~448]{Pikhurko2014}}]
\label{lem:A-config}
Let $A=\{a_1<a_2<\cdots\}\subseteq\N$ be infinite, and let $A$-configurations be as in Construction~\ref{con:A-config}. If $\mathcal F_A$ is the family of all finite $3$-graphs that do not embed into any $A$-configuration, then $\pi(\mathcal F_A)=\Lambda_A$.
\end{lemma}

Let us spell out the correspondence with the cited argument. In the proof of \cite[Theorem~2]{Pikhurko2014}, Pikhurko defines $A$-configurations for an arbitrary infinite set $A=\{a_1<a_2<\cdots\}$, allowing the graph inserted at level $i$ to be $H_{a_i}$. The forbidden family there is precisely the family of finite $3$-graphs not embeddable into the hereditary closure of these recursive configurations. The lower bound comes from finite truncations and blow-ups of the $A$-configuration, while the upper bound is encoded by the identity $\ex(n,\mathcal F_A)=p_{A,n}$, where $p_{A,n}$ is the maximum number of edges in an $n$-vertex $A$-configuration. The limiting value of $p_{A,n}/\binom n3$ is the corresponding recursive Lagrangian, which in our notation is $\Lambda_A$.

\begin{proof}[Proof of Theorem~\ref{thm:interval-3}]
We first show that every point of $(\mu,\nu_1]$ is some $\Lambda_A$ with $A$ infinite. The endpoint $\nu_1$ is obtained from $A=\N$, so let $y\in(\mu,\nu_1)$. We construct $A$ greedily. Suppose that after deciding membership of $1,\ldots,j-1$, the current set is $C\subseteq[j-1]$ and
\begin{equation}
        \Lambda_C<y\le \Lambda_{C\cup T_j}.
        \label{eq:greedy-invariant}
\end{equation}
This holds initially for $j=1$, since $\Lambda_\emptyset=\mu<y\le \nu_1$. At step $j$, include $j$ in $C$ if and only if $\Lambda_{C\cup\{j\}}<y$. If we include $j$, then the lower bound in \eqref{eq:greedy-invariant} holds by definition, and the upper bound at the next step is just the old upper bound, because $(C\cup\{j\})\cup T_{j+1}=C\cup T_j$. If we do not include $j$, then the lower bound is unchanged, while the upper bound follows from \eqref{eq:overlap-property}: $y\le \Lambda_{C\cup\{j\}}\le \Lambda_{C\cup T_{j+1}}$. Thus~\eqref{eq:greedy-invariant} remains valid for all $j$.

Let $A$ be the final set. It is infinite. Indeed, if the process stopped adding elements after some finite set $C_0$, then $\Lambda_{C_0}<y$. Since $\beta_j$ tends to $\alpha$ as $j$ tends to infinity and $\psi_\alpha(\mu)=\mu$, we have
\[
        \lim_{j\to\infty}\psi_{\beta_j}(\mu)=\psi_\alpha(\mu)=\mu .
\]
The finite composition $\Psi_{C_0}$ is continuous, and therefore
\[
        \lim_{j\to\infty}\Lambda_{C_0\cup\{j\}}
        =\lim_{j\to\infty}\Psi_{C_0}(\psi_{\beta_j}(\mu))
        =\Psi_{C_0}(\mu)
        =\Lambda_{C_0}.
\]
Thus every sufficiently large $j$ would satisfy $\Lambda_{C_0\cup\{j\}}<y$ and would be accepted, a contradiction.

Let $C_j\coloneqq A\cap[j-1]$. From the invariant~\eqref{eq:greedy-invariant}, we have 
\begin{equation}
        \Lambda_{C_j}\le y\le \Lambda_{C_j\cup T_j}.
        \label{eq:squeeze1}
\end{equation}
Moreover, since the composition $\Psi_{C_j}$ is $1$-Lipschitz in its input, we have 
\begin{equation}
        0\le \Lambda_{C_j\cup T_j}-\Lambda_{C_j}
        =\Psi_{C_j}(\nu_j)-\Psi_{C_j}(\mu)
        \le \nu_j-\mu .
        \label{eq:squeeze-width}
\end{equation}
By \eqref{eq:nuj-to-mu}, the last quantity tends to $0$ as $j$ tends to infinity.

We now make the squeezing of $\Lambda_A$ explicit. The tail $A\cap T_j$ gives a recursive value between the all-$\alpha$ tail and the full tail $T_j$. Indeed, if $A\cap T_j=\{a_1<a_2<\cdots\}$, then $a_s\ge j+s-1$, and hence $\alpha\le \beta_{a_s}\le \beta_{j+s-1}$ for $s\ge 1$. Monotonicity gives $\mu\le \Lambda_{A\cap T_j}\le \nu_j$. Applying the monotone finite composition $\Psi_{C_j}$, we obtain
\[
        \Lambda_{C_j}=\Psi_{C_j}(\mu) \le \Lambda_A \le \Psi_{C_j}(\nu_j)=\Lambda_{C_j\cup T_j}.
\]
Together with \eqref{eq:squeeze1} and \eqref{eq:squeeze-width}, this implies $\Lambda_A=y$.

It remains to recall why these values are Tur\'an densities. By Lemma~\ref{lem:A-config}, for every infinite $A$ there is an infinite family $\mathcal F_A$ of $3$-graphs such that $\pi(\mathcal F_A)=\Lambda_A$. Therefore every $y\in(\mu,\nu_1]$ belongs to $\Pi^{(3)}_\infty$. In particular, $\nu_j=\Lambda_{T_j}\in\Pi^{(3)}_\infty$ for every $j\ge 1$. By \eqref{eq:nuj-to-mu}, the sequence $(\nu_j)$ tends to $\mu$; since $\Pi^{(3)}_\infty$ is closed by \cite[Proposition 1]{Pikhurko2014}, we also have $\mu\in\Pi^{(3)}_\infty$. Thus $[\mu,\nu_1]\subseteq\Pi^{(3)}_\infty$.

The interval is non-degenerate because $\nu_1\ge \psi_{\beta_1}(\mu)\ge \mu+R\eps_1>\mu$, where the last inequality follows by evaluating $\psi_{\beta_1}(\mu)$ at $\mathbf{z}^*$.
\end{proof}

\section{Consequences and uniformly dense analogues}
\label{sec:extensions-uniform-analogue}
We now derive the remaining consequences from the interval obtained in Section~\ref{sec:recursive-interval}. First, a complete-join operation turns any non-degenerate interval of $r$-uniform Tur\'an densities into an interval ending at $1$. Second, a standard lifting construction transfers the $3$-uniform interval to all higher uniformities. Finally, we use finite palettes to obtain the corresponding dense interval for finite-family uniform Tur\'an densities.

\subsection{Complete joins and terminal intervals}
\label{sec:complete-joins}
In this subsection, we prove the following result. 

\begin{theorem}\label{thm:terminal-from-interval} 
Let $r\ge 3$. Suppose that $\Pi^{(r)}_\infty$ contains a non-degenerate interval $[a,b]\subseteq\Pi^{(r)}_\infty$ for some $a < b$. Then there exists $\delta>0$ such that $[1-\delta,1]\subseteq\Pi^{(r)}_\infty$. 
\end{theorem}

\begin{construction}\label{con:complete-join} Let $G$ be an $r$-graph and let $M\ge 2$. Take $M$ vertex-disjoint copies $V_1,\ldots,V_M$ of $V(G)$. Inside each $V_i$, put one copy of $G$, and add every $r$-set which is not contained in a single $V_i$. Denote the resulting $r$-graph by $J_M(G)$. \end{construction}

\begin{lemma}\label{lem:complete-join} Let $r\ge 3$, let $x\in\Pi^{(r)}_\infty$, and let $M\ge 2$. Then $1-\frac{1-x}{M^{r-1}}\in\Pi^{(r)}_\infty$. \end{lemma}

\begin{proof}
By Theorem~\ref{thm:tight}, there is a sequence of $r$-graphs $G_n$, with $N_n\coloneqq v(G_n)\to\infty$, such that $\varrho(G_n)\to x$ and $\lambda(G_n)\to x$. Let $J_M(G_n)$ be the complete join from Construction~\ref{con:complete-join}, with parts $V_1,\ldots,V_M$.
The only missing $r$-sets of $J_M(G_n)$ are the missing $r$-sets inside the $M$ internal copies. Hence
\[
        |J_M(G_n)| =\binom{MN_n}{r}-M\binom{N_n}{r}+M |G_n|,
\]
and therefore
\begin{equation}
        \lim_{n\to\infty}\varrho(J_M(G_n)) = 1-\frac{1-x}{M^{r-1}}.
        \label{eq:join-density}
\end{equation}

It remains to prove the same limit for the Lagrangian. Let $\mathbf{y}$ be a probability vector on $V(J_M(G_n))$, and put $w_i\coloneqq \sum_{v\in V_i}y_v$. The contribution to $P_{J_M(G_n)}(\mathbf{y})$ from all cross $r$-sets is at most $1-\sum_{i\in[M]} w_i^r$: this is the weight of all ordered $r$-tuples that are not fully contained in one part, allowing repetitions, and hence dominates the normalized weight of the cross $r$-sets. The internal contribution to $P_{J_M(G_n)}(\mathbf{y})$ is at most $\lambda(G_n)\sum_{i\in[M]} w_i^r$. Therefore
\[
        P_{J_M(G_n)}(\mathbf{y}) \le 1-(1-\lambda(G_n))\sum_{i\in[M]} w_i^r.
\]
By convexity,
\[
        \sum_{i\in[M]} w_i^r\ge M\left(\frac1M\right)^r=\frac1{M^{r-1}}.
\]
Thus
\[
        \lambda(J_M(G_n))\le 1-\frac{1-\lambda(G_n)}{M^{r-1}}.
\]
Taking the limit superior gives
\[
        \limsup_{n\to\infty}\lambda(J_M(G_n))\le 1-\frac{1-x}{M^{r-1}}.
\]
For the reverse inequality, evaluate the Lagrangian at the uniform vector on $V(J_M(G_n))$. This gives
\[
        \lambda(J_M(G_n)) \ge \frac{r!\,|J_M(G_n)|}{(MN_n)^r} =\varrho(J_M(G_n))\frac{(MN_n)_r}{(MN_n)^r},
\]
where $(N)_r\coloneqq N(N-1)\cdots(N-r+1)$. Since $MN_n\to\infty$, \eqref{eq:join-density} gives
\[
        \liminf_{n\to\infty}\lambda(J_M(G_n))\ge 1-\frac{1-x}{M^{r-1}}.
\]
Theorem~\ref{thm:tight} now gives the claim.
\end{proof}

\begin{proof}[Proof of Theorem~\ref{thm:terminal-from-interval}]
For each integer $M\ge 2$, Lemma~\ref{lem:complete-join} gives
\[
        I_M\coloneqq \left[ 1-\frac{1-a}{M^{r-1}}, 1-\frac{1-b}{M^{r-1}} \right] \subseteq\Pi^{(r)}_\infty .
\]
Put $A\coloneqq 1-a$ and $B\coloneqq 1-b$. Since $a<b$, we have $0\le B<A$. The intervals $I_M$ and $I_{M+1}$ overlap exactly when $1-\frac{B}{M^{r-1}}\ge 1-\frac{A}{(M+1)^{r-1}}$, or equivalently when $\frac BA\le \left(\frac{M}{M+1}\right)^{r-1}$. 
Let
\[
        \gamma\coloneqq \left(\frac BA\right)^{1/(r-1)}
        \quad\text{and}\quad
        M_0\coloneqq \max\left\{2,\left\lceil\frac{\gamma}{1-\gamma}\right\rceil\right\}.
\]
Then $0\le \gamma<1$, and $M_0$ is the smallest integer $M\ge 2$ such that the intervals $I_M,I_{M+1},I_{M+2},\ldots$ overlap consecutively. Indeed, for an integer $M\ge 2$, the overlap condition is equivalent to $\gamma\le \frac{M}{M+1}$, which is equivalent to $M\ge \gamma/(1-\gamma)$.
Thus
\[
        \bigcup_{M\ge M_0} I_M
        \supseteq \left[1-\frac{A}{M_0^{r-1}},1\right).
\]
Since $1\in\Pi^{(r)}_\infty$, we may take $\delta\coloneqq \frac{A}{M_0^{r-1}}=\frac{1-a}{M_0^{r-1}}$. 
Then $[1-\delta,1]\subseteq\Pi^{(r)}_\infty$, as required.
\end{proof}

\subsection{Lifting to higher uniformities}
\label{sec:lifting}
In this subsection we lift the preceding $3$-uniform result to all higher uniformities. For $r\ge 3$, put
\[
        c_r\coloneqq \frac{3^3 r!}{3!\,r^r}=\frac{9r!}{2r^r}.
\]

\begin{lemma}\label{lem:uniformity-lift} 
Let $r\ge 3$. If $x\in\Pi^{(3)}_\infty$, then $c_rx\in\Pi^{(r)}_\infty$. Consequently, if $[a,b]\subseteq\Pi^{(3)}_\infty$, then $[c_ra,c_rb]\subseteq\Pi^{(r)}_\infty$. 
\end{lemma}

\begin{construction}[Uniformity lifts]\label{con:uniformity-lift} 
Let $r\ge 3$ and let $G$ be a $3$-graph. If $r=3$, set $L_r(G)\coloneqq G$. If $r>3$, put $s\coloneqq r-3$ and $N\coloneqq v(G)$, and let $\ell\coloneqq \lfloor N/3\rfloor$. Define the $r$-graph $L_r(G)$ as follows. Its vertex set is $V(G)\sqcup U_1\sqcup\cdots\sqcup U_s$ with $|U_1|=\cdots=|U_s|=\ell$, and its edges are exactly the sets $e\cup\{u_1,\ldots,u_s\}$ for $e \in G$ and $u_i \in U_i$. 
\end{construction}

\begin{proof}[Proof of Lemma~\ref{lem:uniformity-lift}]
The case $r=3$ is trivial, so assume $r>3$ and set $s\coloneqq r-3$. By Theorem~\ref{thm:tight}, there is a sequence of $3$-graphs $G_n$, with $N_n\coloneqq v(G_n)\to\infty$, such that $\varrho(G_n)\to x$ and $\lambda(G_n)\to x$.

Let $H_n\coloneqq L_r(G_n)$ be the lift from Construction~\ref{con:uniformity-lift}, with auxiliary parts $U_1,\ldots,U_s$, and put $\ell_n\coloneqq \lfloor N_n/3\rfloor$. Thus $|H_n|=|G_n|\ell_n^s$. Since $\frac{N_n}{N_n+s\ell_n}\to \frac3r$ and $\frac{\ell_n}{N_n+s\ell_n}\to \frac1r$, we get
\begin{align*}
        \varrho(H_n)
        &=
        \frac{|G_n|\ell_n^s}{\binom{N_n+s\ell_n}{r}}  
        =
        \varrho(G_n)\frac{\binom{N_n}{3}\ell_n^s}{\binom{N_n+s\ell_n}{r}}.
\end{align*}
Taking $n$ to infinity in the last expression gives
\[
        \lim_{n\to\infty}\varrho(H_n)
        =
        x\cdot \frac{r!}{3!}\left(\frac3r\right)^3\left(\frac1r\right)^s
        =c_rx.
\]

It remains to compute the Lagrangian. Let $\mathbf{y}$ be a probability vector on $V(H_n)$. Write $A\coloneqq \sum_{v\in V(G_n)}y_v$ and $B_i\coloneqq \sum_{u\in U_i}y_u$. If $A>0$, let $\widehat{\mathbf{y}}$ be the normalized restriction of $\mathbf{y}$ to $V(G_n)$. Then
\begin{align*}
        P_{H_n}(\mathbf{y})
        &=
        r!\sum_{e\in G_n}\prod_{v\in e}y_v\prod_{i=1}^s B_i  
        =
        \frac{r!}{3!}A^3\prod_{i=1}^s B_i\,P_{G_n}(\widehat{\mathbf{y}})  
        \le
        \frac{r!}{3!}A^3\prod_{i=1}^s B_i\,\lambda(G_n).
\end{align*}
The same inequality is trivial when $A=0$. Under the constraint $A+B_1+\cdots+B_s=1$, the product $A^3B_1\cdots B_s$ is maximized at
\[
        A=\frac3r,\qquad B_1=\cdots=B_s=\frac1r.
\]
Hence
\[
        \lambda(H_n) \le \frac{r!}{3!}\left(\frac3r\right)^3\left(\frac1r\right)^s\lambda(G_n) =c_r\lambda(G_n).
\]
Equality is obtained by taking a $G_n$-optimal vector on $V(G_n)$, scaling it by $3/r$, and putting total mass $1/r$ on one vertex of each $U_i$. Therefore $\lambda(H_n)=c_r\lambda(G_n)\to c_rx$. Thus $H_n$ is asymptotically Lagrangian-tight with limiting density $c_rx$. By Theorem~\ref{thm:tight}, $c_rx\in\Pi^{(r)}_\infty$.
\end{proof}

Theorem~\ref{thm:main} is now a direct consequence of Theorem~\ref{thm:interval-3}, Lemma~\ref{lem:uniformity-lift}, and Theorem~\ref{thm:terminal-from-interval}. 

\begin{proof}[Proof of Theorem~\ref{thm:main}]
By Theorem~\ref{thm:interval-3}, there is a non-degenerate interval $[\mu,\nu_1]\subseteq\Pi^{(3)}_\infty$. Let $r\ge 3$. By Lemma~\ref{lem:uniformity-lift}, $[c_r\mu,c_r \nu_1]\subseteq\Pi^{(r)}_\infty$. This interval is non-degenerate because $c_r>0$ and $\nu_1>\mu$. Applying Theorem~\ref{thm:terminal-from-interval} in uniformity $r$ gives a number $\delta_r>0$ such that $[1-\delta_r,1] \subseteq\Pi^{(r)}_\infty$.
\end{proof}

\subsection{Dense intervals of finite-family uniform Tur\'an densities}
\label{sec:dense-uniform-interval}
In this subsection, we briefly explain how to obtain Corollary~\ref{cor:dense-uniform-finite-interval}.
We first fix the palette notation. A finite \emph{palette} $P$ consists of a finite color set $C(P)$ and a set $\mathcal T(P)\subseteq C(P)^3$ of admissible ordered triples. 
The \emph{finite palette Lagrangian} of $P$ is
\[
        \Lambda(P) \coloneqq \max_{\mathbf{x}\in\Delta_{C(P)}} \sum_{(a,b,c)\in\mathcal T(P)} x_a x_b x_c .
\]
Recall that $\Delta_{C(P)}$ consists of all probability vectors on $C(P)$. Let
\[
        \Lambda^{(3)}_{\mathrm{pal}}\coloneqq \{\Lambda(P)\colon P\text{ is a finite palette}\}.
\]

The external input we use here is the following theorem of King, Piga, Sales and Sch\"ulke~\cite{KingPigaSalesSchulke2025}: every Lagrangian of a finite palette is the uniform Tur\'an density of some finite family of $3$-graphs. For completeness, we include the brief explanation showing how $3$-graph Lagrangians are realized as finite palette Lagrangians.

\begin{lemma}[\cite{KingPigaSalesSchulke2025}]\label{lem:graph-lagrangian-palettes}
For every $k\in\{1,2,3,4,5,6\}$, we have $\frac{k}{6}\Lambda^{(3)}\subseteq \Lambda^{(3)}_{\mathrm{pal}} \subseteq \Pi_{\vvv,\fin}$. 
\end{lemma}

\begin{proof}
The second inclusion $\Lambda^{(3)}_{\mathrm{pal}}\subseteq\Pi_{\vvv,\fin}$ is precisely the theorem of King, Piga, Sales and Sch\"ulke~\cite[Theorem~1.1]{KingPigaSalesSchulke2025}.

Let us justify the first inclusion.  Fix a finite $3$-graph $G$ and $k\in\{1,\ldots,6\}$. Let $T\subseteq S_3$ be any set of $k$ permutations. Define a finite palette $P=P(G,k)$ with color set $C(P)=V(G)$ as follows. For every edge $e=\{a,b,c\}\in G$, choose once and for all an ordering $(v_1,v_2,v_3)$ of its vertices, and include exactly the ordered triples
\[
        (v_{\sigma(1)},v_{\sigma(2)},v_{\sigma(3)}) \qquad\text{for all } \sigma\in T
\]
as admissible triples corresponding to $e$. These $k$ ordered triples are distinct because the vertices of $e$ are distinct.

For every $\mathbf{x}\in\Delta_{V(G)}$, the palette polynomial satisfies
\[
        \sum_{(i,j,\ell)\in\mathcal T(P)}x_ix_jx_\ell
        =
        k\sum_{\{u,v,w\}\in G}x_ux_vx_w.
\]
Taking the maximum over all $\mathbf{x}\in\Delta_{V(G)}$ gives
\[
        \Lambda(P)
        =
        k\max_{\mathbf{x}\in\Delta_{V(G)}}\sum_{\{u,v,w\}\in G}x_ux_vx_w
        =
        \frac{k}{6}\lambda(G),
\]
where the last equality uses the normalized Lagrangian convention from Section~\ref{sec:preliminaries}. Since $G$ was arbitrary, the claim follows.
\end{proof}

\begin{proof}[Proof of Corollary~\ref{cor:dense-uniform-finite-interval}]
By Theorem~\ref{thm:interval-3}, there is a non-degenerate interval $[\mu,\nu_1]\subseteq\Pi^{(3)}_\infty$. By Theorem~\ref{thm:pikhurko-lagrangian-densities}, we have $\overline{\Lambda^{(3)}}=\Pi^{(3)}_\infty$. Hence $\Lambda^{(3)}$ is dense in $[\mu,\nu_1]$.
By Lemma~\ref{lem:graph-lagrangian-palettes}, we have $\Lambda^{(3)}\subseteq\Pi_{\vvv,\fin}$. Thus $\Pi_{\vvv,\fin}$ is dense in $[\mu,\nu_1]$, and therefore $[\mu,\nu_1]\subseteq\overline{\Pi_{\vvv,\fin}}$. 
This proves the corollary.
\end{proof}

\section{Concluding remarks}
\label{sec:concluding-remarks}
$\bullet$ Although our terminal interval lies far above the expected first non-jump, it is in line with the conjectural picture proposed by  Mubayi and the first author~\cite{LiuMubayi2026FourNinths}. They conjecture that, for $3$-graphs, every number in $[0,4/9)$ is a jump and every number in $[4/9,1)$ is a non-jump; our interval near $1$ supports the idea that non-jump behaviour may occur on whole intervals rather than only at isolated values.

\medskip 

$\bullet$ Our proof uses a stronger form of non-jump behaviour. Let us call a point $x\in[0,1)$ a \emph{smooth non-jump point} if there are finite $r$-graphs $G_i$ such that
\[
        \lambda(G_i)>x,\qquad \lambda(G_i)\downarrow x,\qquad\text{and}\qquad
        \frac{\lambda(G_i)-x}{\lambda(G_{i+1})-x}\to 1.
\]
Thus the possible densities approach $x$ from above with no asymptotically visible gaps between successive approximants. Our construction gives such a point for $r=3$, and the rest of the proof shows how one smooth non-jump point can be amplified into intervals of possible Tur\'an densities.

\begin{problem}
How common are smooth non-jump points? Does every non-jump arise, or can every non-jump be approximated, in this stronger smooth sense?
\end{problem}

\medskip 

$\bullet$ Although we have made no attempt to optimize the constants, the proof is effective in principle. For instance, in the $3$-uniform case one may choose explicit parameters in the spectral construction, say
\[
        p=137,\qquad D=p+1=138,\qquad k=30,\qquad m=2701 .
\]
With these choices the parameter $\alpha$ is $899/900$, and the finite-dimensional optimization defining $\mu$ reduces, by symmetry, to a two-variable rational optimization problem after symmetrizing the coordinates $3,\ldots,m$. This can be certified by interval arithmetic; numerically one obtains
\[
        \mu = 0.9988895745058906452116222183898284685\ldots .
\]
The remaining quantities $\nu_1-\mu$ and the terminal interval constant $\delta_3$ depend on the effective choice of the auxiliary LPS primes in a fixed arithmetic progression (this is where the dominant loss comes from). Using explicit estimates for primes in arithmetic progressions, for example the bounds of Bennett, Martin, O'Bryant and Rechnitzer~\cite{BennettMartinOBryantRechnitzer2018}, and then choosing a sufficiently far tail of primes $q\equiv 1\pmod {548}$, gives an explicit positive lower bound on $\nu_1-\mu$. Combined with the complete join argument of Subsection~\ref{sec:complete-joins}, this yields an explicit value of $\delta_3>0$ for which $[\,1-\delta_3,1\,]\subseteq \Pi_\infty^{(3)}$. 
The resulting numerical value is extremely small and is not intended to be sharp; the point is only that the construction can be made fully effective.

In relation to Corollary~\ref{cor:dense-uniform-finite-interval}, the obvious  open problems are to show that every value in some non-degenerate interval is a uniform Tur\'an density of $3$-graphs and to extend it to arbitrary uniformity $r\ge 4$.

\section*{Acknowledgement}

Xizhi Liu was supported by the Excellent Young Talents Program (Overseas) of the National Natural Science Foundation of China. Oleg Pikhurko was supported by ERC Advanced Grant 10102025.

\section*{Declaration on the use of generative AI}

The beautiful idea of using Ramanujan graphs for the internal spectral gadget was suggested by generative AI tools (ChatGPT 5.4 Pro and 5.4 Thinking). The other main ideas and constructions in the paper originate from the authors' project dating back to around 2023.

\bibliographystyle{abbrv}
\bibliography{Turan}
\end{document}